\documentclass[final,smallextended]{svjour3}
\smartqed
\usepackage{mathptmx}
\usepackage{amsmath}
\usepackage{amsfonts}

\ifx11

\usepackage{tikz}
\usetikzlibrary{calc}
\usetikzlibrary{arrows}   %

\else
\usepackage{color}
\definecolor{grey}{rgb}{0.5,0.5,0.5}
\definecolor{gray}{rgb}{0.5,0.5,0.5}
\usepackage[final]{graphicx}
\long\def\beginpgfgraphicnamed#1#2\endpgfgraphicnamed{\message{Including file instead of processing TikZ picture.}\includegraphics{#1}}
\fi

\let\JKifAnot=\iffalse

\JKifAnot
\newcommand \NewB{\marginpar{\tikz \fill[orange] (1ex,1ex) circle (1ex); new[\let\eatthischar]}}
\newcommand \NewE{\marginpar{\let\eatthischar[] \tikz \fill[gray] (1ex,1ex) circle (1ex);}}
\newcommand \NewBrefsug{\marginpar{\tikz \fill[red] (1ex,1ex) circle (1ex); new[\let\eatthischar]}}
\newcommand \NewErefsug{\marginpar{\let\eatthischar[] \tikz \fill[black] (1ex,1ex) circle (1ex);}}

\newcommand \NewBsmall{\bgroup\color{red}}
\newcommand \NewEsmall{\egroup}
\newcommand \NewNewBsmall{\bgroup\color{blue}}
\newcommand \NewNewEsmall{\egroup}
\newcommand \NewBrefsugSmall{\bgroup\color{red!70!blue}}
\newcommand \NewErefsugSmall{\egroup}
\newcommand \NewChangeSmall{{\color{orange}$\bullet_{\text{ch}}$}}
\newcommand \NewNewRmSmall{{\color{blue}$\bullet_{\text{rm}}$}}
\newcommand \NewMovedHereSmallB{{\color{orange}$\bullet_{\text{mv}}^{\mathbb [}$}} %
\newcommand \NewMovedHereSmallE{{\color{orange}$\bullet_{\text{mv}}^{\mathbb ]}$}} %
\newcommand \NewMovedSmallL{{\color{orange}$\bullet_{\text{mv}}^\leftarrow$}} %

\else
\newcommand \NewB{}
\newcommand \NewE{}
\newcommand \NewBrefsug{}
\newcommand \NewErefsug{}
\newcommand \NewBsmall{}
\newcommand \NewEsmall{}
\newcommand \NewNewBsmall{}
\newcommand \NewNewEsmall{}
\newcommand \NewBrefsugSmall{}
\newcommand \NewErefsugSmall{}
\newcommand \NewChangeSmall{}
\newcommand \NewNewRmSmall{}
\newcommand \NewMovedHereSmallB{} %
\newcommand \NewMovedHereSmallE{} %
\newcommand \NewMovedSmallL{} %

\fi

        \usepackage{hyperref}
        \hypersetup{
                                     colorlinks=true,
                                     linkcolor={black!85!white},
                                     citecolor={black!85!white},
                                     pdfborder={0 0 0},
                                   }

\newcommand{\JKXallowbreak}{\penalty 9999\relax}%

\newcommand\qqeedd{\quad\hfill\hbox{\rlap{$\llcorner$}$\urcorner$}}

\newcommand*\indn[1][]{^{(n#1)}}
\newcommand*\indX[1][n]{^{(#1)}}
\newcommand\RAD{\textup{rad}}
\newcommand\JOneThree{J_{13}^{24}}
\newcommand\JOneTwo {J_{12}^{34}}
\newcommand\RadAndOneThree{\RAD\;\& \JOneThree}
\newcommand\RadAndOneTwo {\RAD\;\& \JOneTwo}
\newcommand*\RadAndX[1]{\RAD\;\& #1}

\newcommand{\ImageMeasureFM}{\#}
\newcommand{\ImageMeasureFV}{\#}
\providecommand{\ImageMeasureFM}{{\#_{mea}}}
\providecommand{\ImageMeasureFV}{{\#_{var}}}
\newcommand{\ImageVarifold}{{\#\!\#}}

\newcommand{\mrest}{\mathbin{\llcorner}}
\newcommand{\R}{\mathbb R}
\newcommand{\Complex}{\mathbb C}
\newcommand{\N}{\mathbb N}
\newcommand{\Z}{\mathbb Z}
\DeclareMathOperator{\Span}{span}
\DeclareMathOperator{\spt}{spt}

\DeclareMathOperator{\Log}{Log}
\DeclareMathOperator{\diver}{div}
\DeclareMathOperator{\sgn}{sgn}
\DeclareMathOperator{\range}{range}
\DeclareMathOperator{\dom}{dom}
\DeclareMathOperator{\mass}{\mathbf M}
\DeclareMathOperator{\VarTan}{Var\,Tan}
\DeclareMathOperator{\TanGeneral}{Tan}
\DeclareMathOperator{\inter}{int}
\newcommand{\Lebesgue}{\mathcal L}
\newcommand{\Hausd}{\mathcal H}
\newcommand{\Continuous}{\mathcal C}

\newcommand{\Sphere}{\mathbb S}
\newcommand{\Sone}{\Sphere^1}
\newcommand{\setcolon}{:}

\newcommand{\vari}{\delta}
\newcommand{\intd}{\,\mathrm{d}}  %
\newcommand{\intdx}{\mathrm{d}}   %
\newcommand{\Mucircles}{\mathcal M}

\spnewtheorem*{interpretation}{Interpretation.}{\it}{\rm}
\spnewtheorem*{solution*}{Solution.}{\it}{\rm}
\numberwithin{proposition}{section}
\numberwithin{theorem}{section}
\numberwithin{lemma}{section}
\numberwithin{exercise}{section}
\numberwithin{remark}{section}

\begin{document}
\title{Non-unique conical and non-conical tangents to rectifiable stationary varifolds in $\mathbb R^4$\thanks{Research
supported by grants P201/12/0290 of GA\,\v{C}R, IAA100190903 of GA\,AV and RVO: 67985840.}}

\journalname{submitted to Calculus of Variations and Partial Differential Equations --- }
\author{Jan Kol\'a\v{r}}
\institute{
 Institute of Mathematics
        \\
 Czech Academy of Sciences
        \\
 \v{Z}itn\'a 25
        \\
 115 67 Praha 1
        \\
 Czech Republic
        \\
        \email{kolar@math.cas.cz}%
}
\date{December 30, 2012 \ \ revised May 15, 2014 \ \ \ \ \ \ \ Received: date / Accepted: date}
\maketitle

\begin{abstract}
  We construct
  \NewMovedHereSmallB
  a rectifiable
  \NewMovedHereSmallE
  stationary $2$-varifold in $\mathbb R^4$
  with non-conical,
  and hence non-unique,
  tangent varifold at a point.
  This answers a question of
  L.~Simon
  (Lectures on geometric measure theory, 1983, p.~243)
  and
    \NewBsmall
    provides a new example for
    \NewEsmall
  a related question of W.\,K. Allard
  (On the first variation of a varifold,
   Ann. of Math.,
  1972, p.~460).
  \NewMovedSmallL

  There is also a (rectifiable) stationary $2$-varifold in $\mathbb R^4$
  that has more than one conical tangent varifold at a point.
\end{abstract}

\tableofcontents

\section{Introduction}

 \NewB

\subsection{General context}

  Geometrical measure theory uses various ``generalized surfaces'' to reach its goals,
  and the varifolds are among them.
  Most of them allow,
  i.a.,
  countably many pieces of surface that are interconnected into simple or complicated networks
  (Figure~\ref{fig:network}).
  The classes of surfaces are designed to have compactness properties and to allow to obtain a generalized surface
  of least area among those that, say, span a given boundary.
  The next, equally important, step is to explore smoothness and regularity properties of the minimizer.

  {\em Uniqueness of tangents} is both an important attribute encompassed in various definitions of smoothness and regularity,
  and an important tool. Here a tangent is (informally) defined as a limit of a sequence of blow-ups at a given point.
  The tangents implicitly appear already in the basic calculus of real functions:
  A Lipschitz function on $\R^n$ is differentiable at a point $x_0$ if and only if it admits a unique tangent
  \NewNewBsmall
  at $x_0$ and the tangent is a hyperplane.
  \NewNewEsmall
  For $n=1$, the existence of the two one-sided derivatives at $x_0$ is equivalent to uniqueness of the tangent at $x_0$
  and the tangent is then necessarily a cone.

  \message{--figure--}
  \def\rrrr{1.8}
  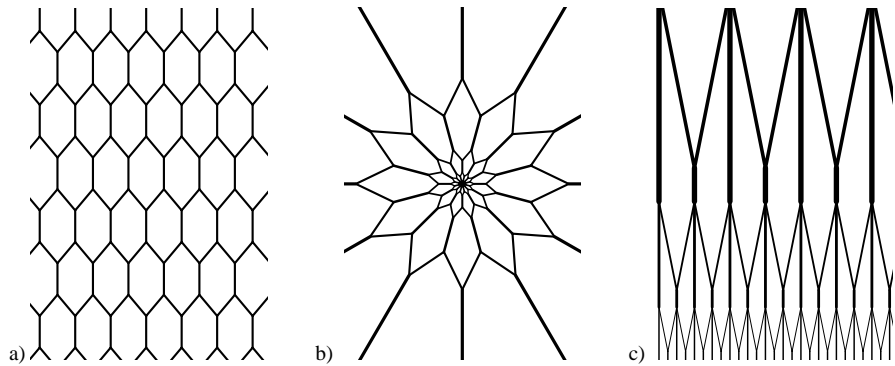
\begin{figure}[hbt]
  \begin{center}
  \def\SCALE{56mm/36/2}
  \beginpgfgraphicnamed{varif2012-fig-network}
  \def\SCALEAAxx{56mm/36/10 }
  \def\SCALEAAyy{56mm/8}
  a)
  \begin{tikzpicture}[x=\SCALEAAxx,y=\SCALEAAyy, line cap=round]
    \relax
  \clip
  [xscale={(\SCALE)/(\SCALEAAxx)},yscale={(\SCALE)/(\SCALEAAyy)}, shift={(1.5,1.5)}]
  (0,0) rectangle (2*20,2*30);
  \def\thickness{2.4pt}
  \def\thicknessA{\thickness/3} %
  \def\thicknessB{\thickness/3} %
  \foreach \YY in {8, 6, 4, 2}
  {
  \foreach \angle in {0, 30,..., 355}
  {
  \draw [line width={\thicknessA}]
    (\angle,\YY) -- (\angle,{\YY-0.6});
  }
  \foreach \angle in {0, 30,..., 355}
  {
  \draw [line width={\thicknessB}]
    (\angle,{\YY-0.6}) -- ({\angle+15},{\YY-1.0})
    ({\angle+30},{\YY-0.6}) -- ({\angle+15},{\YY-1.0})
    ;
  \draw [line width={\thicknessA}]
    ({\angle+15},{\YY-1.0}) -- ({\angle+15},{\YY-1.6})
    ;
  \draw [line width={\thicknessB}]
    ({\angle+15},{\YY-1.6}) -- ({\angle},{\YY-2.0})
    ({\angle+15},{\YY-1.6}) -- ({\angle+30},{\YY-2.0})
    ;
  }
  }
  \end{tikzpicture}
  \qquad
  b)
  \message{--tikzpicture--}
  \begin{tikzpicture}[x=\SCALE,y=\SCALE, line cap=round]
    \relax
  \clip
  (-20,-30) rectangle (20,30);
  \foreach \RR/\thickness in {36/2.4pt, 4/{2.4pt/1.5625}, {4/9}/{2.4pt/1.953125}}
  {
  \foreach \angle in {0, 30,..., 355}
  {
  \draw [line width={\thickness/2}]
    (\angle:\RR) -- (\angle:{\RR/2});
  }
  \foreach \angle in {0, 30,..., 355}
  {
  \draw [line width={\thickness/3}]
    (\angle:{\RR/2}) -- ({\angle+15}:{\RR/3})
    ({\angle+30}:{\RR/2}) -- ({\angle+15}:{\RR/3})
    ;
  \draw [line width={\thickness/2.5}]
    ({\angle+15}:{\RR/3}) -- ({\angle+15}:{\RR/6})
    ;
  \draw [line width={\thickness/3.66}]
    ({\angle+15}:{\RR/6}) -- ({\angle}:{\RR/9})
    ({\angle+15}:{\RR/6}) -- ({\angle+30}:{\RR/9})
    ;
  }
  }
  \end{tikzpicture}
  \qquad
  c)
\def\DOpictureROVNY#1#2#3#4#5#6#7{\relax
  \def\RR{#1}\relax
  \def\step{#2}\relax
  \def\Rmult{#3}\relax
  \def\start{#4}\relax
  \def\stopA{#5}\relax
  \def\stopB{#6}\relax
  \def\forDATA{#7}\relax
  \pgfmathparse{0.82*exp(ln(\Rmult)*0.5)}\relax
  \xdef\RmultQQQ{\pgfmathresult}\relax
  \message{--tikzpicture--}\relax
  \begin{tikzpicture}[x=\SCALEAAxx,y=\SCALEAAyy, line cap=round]
  \relax
  \clip
  [xscale={(\SCALE)/(\SCALEAAxx)},yscale={(\SCALE)/(\SCALEAAyy)},
   shift={(-0.5,9.0)}]
  (0,0) rectangle (2*20,2*30);
  \relax
  \def\thickn{4pt}
  \forDATA
  {
  \pgfmathparse{\start+\step}\xdef\startstep{\pgfmathresult}
  \foreach \angle in {\start, \startstep,..., \stopA}
  {
  \draw [line width={\thickness/2}]
    (\angle,\RR) -- (\angle,{\RR*\Rmult})
    (\angle+\step/2,\RR*\RmultQQQ) -- (\angle+\step/2,{\RR*\Rmult});
  }
  \foreach \angle in {\start, \startstep,..., \stopB}
  {
  \draw [line width={\thickness/3}]
    (\angle,{\RR/1})  -- ({\angle+\step/2},{\RR*\RmultQQQ})
    ({\angle+\step},{\RR/1})  -- ({\angle+\step/2},{\RR*\RmultQQQ})
    ;
  }
  \message{^^J----- \RR : \step : \Rmult : \RmultQQQ ------- }
  \xdef\RR{\RR*\Rmult}
  \pgfmathparse{\step/2}
  \xdef\step{\pgfmathresult}
  }
  \end{tikzpicture}
}
  \DOpictureROVNY{8}{60}{0.5}{0}{209}{209}{\foreach \thickness in {\thickn, {\thickn/1.9}, {\thickn/1.9/1.9}}}\qquad
   \endpgfgraphicnamed
  \caption{Some networks of segments.
  The first two illustrate analogy and differences of linear and central (radial) configurations.
  The idea contained in the last one is actually used in this paper.
  a) The set of weak limits of sequences of vertical upward shifts of the varifold corresponding to this network is uncountable.
  b) With the unit density on each segment we have non-unique tangents but the varifold is not stationary.
  Although it can be converted to a stationary varifold by assigning suitable densities, it does not provide
  a stationary example with non-unique tangents. The densities necessarily converge to zero near the center,
  and the zero varifold is the unique tangent at the centre.
  c) This network is continually branching and refining in the downward direction (towards an interface line). Such a network
  was used by Brakke \cite[p.~238, 240, 250]{Brakke} in the context of varifolds evolving by its `mean curvature'.
  We use its radial variation in a more complicated arrangement.\relax
  \label{fig:network}}
  \end{center}
  \message{--figure--finito--.}
  \end{figure}

  Likewise, a junction of three smooth curves $\gamma_i\setcolon [0,1) \to \R^2$
  at $x_0\in \R^2$ (Figure~\ref{fig:3curves})
  \message{--figure--}
  \def\rrrr{1.8}
  \begin{figure}[hbt]
  \begin{center}
  \beginpgfgraphicnamed{varif2012-fig-spirals}
  \message{--tikzpicture--}
  a)
  \begin{tikzpicture}
  \foreach \angle in {0, 120, 240}
  {
  \draw [very thick]
  (0,0) .. controls (\angle+60:0.4) and (\angle+20:1) .. (\angle:2);
  \draw [very thin]
  (0,0) -- (\angle+60:0.8);
  } ;
  \node [ at={(0,-0.01)}, right] { $x_0$ };
  \end{tikzpicture}
  \qquad
  b)
  \message{--tikzpicture--}
  \begin{tikzpicture}
  \foreach \angle in {0, 120, 240}
  {
   \draw[domain=0:80,samples=100,smooth,variable=\aaa, very thick]
    plot    (\angle+8*\aaa:{\rrrr*exp((\aaa^1 * ln(0.9))})
    ;
  }
  \node [ at={(0.08,0)}, right] { \vtop{\hbox{}\hbox{$x_0$}} };
  \end{tikzpicture}
  \qquad
  c)
  \message{--tikzpicture--}
  \begin{tikzpicture}
  \foreach \angle in {0, 120, 240}
  {
   \draw[domain=0:20,samples=100,smooth,variable=\aaa, very thick]
    plot    (\angle+4*8*\aaa:{\rrrr*exp((\aaa^2 * ln(0.9))})
    ;
  }
  \node [  at={(0,+0.1)}, above ] { $x_0$ };
  \end{tikzpicture}
   \endpgfgraphicnamed
  \caption{Three curves in the plane.
  (Think about this also as a planar section of a hypothetical joint of minimal surfaces in equilibrium.)
  a) Curves smooth up to the end. Unique (and conical) tangent at $x_0$.
  b) The logarithmic spirals. Non-unique tangent. The tangents
  are represented by the rotations of the same picture.
  c) ``Spirals with varying speed'' $r(\alpha)=e^{c(\alpha-\alpha_0)^2}$, $\alpha\ge\alpha_0$.
  The tangents are all $120^\circ$-triples of half-lines, so the
  tangents are conical but non-unique.\label{fig:3curves}}
  \end{center}
  \message{--figure--finito--.}
  \end{figure}
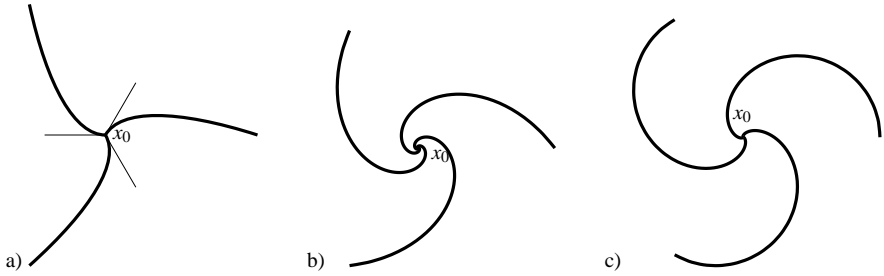
  is considered more regular if the object has
  a
  unique tangent
  \NewNewRmSmall
  at $x_0$
  (the curves have
  \NewNewBsmall
  a non-zero
  \NewNewEsmall
  one-sided derivative at the endpoint).
  In this case they can also be studied as graphs of functions satisfying a differential equation,
  and this can be helpful if they came out of a variational problem.

  The uniqueness of tangents {\em is the regularity}, or a basic degree of the regularity.
  It is an interpretation of what existence of the derivative would be in case we face more general objects
  than graphs of functions.
  In fact, the mathematical language is somewhat inhomogeneous in not having a single word
  for ``the unique tangent'' (of a varifold, e.g.)
  as a counterpart of ``the derivative''.\footnote{Though, in different context the tangent cone is sometimes defined
  to be what we call the unique tangent cone, see for example \cite[p.~159]{Kiselman}.}
  This choice of terminology is not surprising
  since uniqueness of tangents in Geometrical measure theory
  is from the beginnings connected to open problems and later only to partial results.\relax
  \footnote{\relax
          See for example
          \cite[p.~591]{Simon1993}:
          ``... but it is far from obvious (and an open
          question) whether or not $\TanGeneral _X M$ can contain more than one cone $C$ if $X\in \operatorname{sing} M$.''
          The same paper contains a result on the uniqueness of tangents $m$-almost everywhere in the singular set
          \cite[p.~650, (2), (1)]{Simon1993},
          where $m$ is the `top dimension' (e.g., $m=\dim M - 2$, depending on the context).
          }

\medbreak

Now let us give an example of how uniqueness of tangents might be used as a {\em tool}:
It is the result of Sheldon Chang that
the singular set of
area minimizing two-dimensional integral currents
consists of isolated points
and that near any such point their structure
is the same as
that
of a classical branched minimal surface
\cite{Chang}.

Based on the work of B.~White \cite{White1983},
Chang first notes that
(in the case he considers, i.e., the case of Riemannian manifolds)
two dimensional
area minimizing integral currents
have
unique tangent cones
and
he
estimates the rate of pointwise convergence.
He says that this steps are
``necessary for the construction of the first center
manifold.''
\cite[p.~701]{Chang}.

The uniqueness of tangents is also used
in \cite[Chapter~5 and Theorem~0.12]{DLS-Q-revisited}
where an improvement
(in Chang's spirit)
of the size of the singular set of Dir-minimizing Q-valued functions
(on $\Omega \subset \R^2$)
is given.

\subsection{Known results}

The structure of one-dimensional stationary varifolds with density bounded away from zero is
well known \cite{AllAlm1976}.

\smallbreak
 A result about uniqueness of tangent cones of two-dimensional
 soap-bubble-like and soap-film-like minimal surfaces
 ($(M,\xi,\delta)$-minimal sets)
 in $\R^3$
 is contained in \cite{Taylor1976soap}.

Tangent cones to two-dimensional area-minimizing integral currents are unique
by the result of B.~White \cite{White1983}.
As we already noted, this was generalized to Riemannian manifolds by Chang~\cite{Chang}.

\smallbreak

For more general dimensions,
there are results for some special cases, with assumptions related for example to calibration.
Note that the notions of ($\omega$-)positive, (semi\discretionary{-)}{}{-)}calibrated and (pseudo-)holomorphic currents are to a large extent
synonymous (cf.\ \cite{BellettiniTanPosit},  \cite{BellettiniUTC-DukeMathJ}).
Recent results with this kind of assumptions can be found in \cite{PumbergerRiviere2010} ($2$-dimensional), \cite{BellettiniTanPosit}, \cite{BellettiniUTC-DukeMathJ}.
As Bellettini \cite{BellettiniTanPosit} notes, the integrable case $\Complex^n$ of his results follows already by \cite{Siu}.

Very nice result is \cite{Simon1983}, which has a partial generalisation \cite{Simon1994}.
 Simon \cite[Corollary on page 564]{Simon1983} does not assume calibrations.
 The corollary states that if $C$ is a tangent cone to a stationary
 varifold $V$ at a point $p$,
 $C$ has density $1$ on $\spt C \setminus \{ 0 \}$
 (hence $C$ is integral and $0$ is
 the only singular point of $C$)
 then $C$ is the unique tangent cone of $V$ at $p$
 and we have a $\Continuous^2$-flavour of convergence of blowups at $p$.
 \cite{Simon1983} improved earlier result \cite{AA81} which included assumption on integrability of Jacobi fields
 and already covered the case of the cone over the cartesian product of two (but not more, cf.~\cite{White1998-fredall}) standard spheres
 (of arbitrary dimensions).
 \cite{Simon1994} %
 provides similar result where $C=C_0\times \R$ are allowed
 to be certain cases of cylinders with singular set $\{0\} \times \R$.
($C_0$ is assumed to be a strictly minimizing, strictly stable codimension one cone,
and to admit a nice Jacobi-field operator).

        As it can be seen from the above,
        even the codimension one case remains open.
Notably,
\NewNewRmSmall
it remains open whether the hyper-cones
over $\Sphere^3 \times \Sphere^3\times \R$ and $\Sphere^2\times \Sphere^4\times \R$ in $\R^9$
are
always unique tangent cones when they arise at all as multiplicity one tangent cones \cite[p.~1--2]{Simon1994}.
(The question in its formulation in \cite{Simon1994} seemingly concerns the hyper-cones over $\Sphere^3 \times \Sphere^3$ and $\Sphere^2\times \Sphere^4$ in $\R^8$
but that was already solved by \cite[p.~215, (1) and (2)]{AA81}, as well as \cite{Simon1983}.)

Kiselman's example \cite{Kiselman} with non-unique tangent cones is mentioned in the next paragraph.
There is also an example \cite{Kolar} consisting of spirals and a number of lines. It shares with the minimal surfaces
an important property called the ``{\em monotonicity}'' --- for balls centered at an arbitrary fixed point the measure ratio is non-decreassing.
In \cite{CKR:R2} and \cite{CKR:R3} the number of lines is reduced so that the density is, everywhere in the support,
between $1$ and $3+\varepsilon$ (the planar example), or between $1$ and $2+\varepsilon$ (the example in $\R^3$).

\subsection{The questions and the main result}

  The purpose of this paper is to answer a question of L.~Simon \cite[p.~243]{Simon}.
  Simultaneously we provide a new example for a related question of W.\,K.~Allard \cite[p.~460]{Allard1972}.

  Allard's question
  was in a different spirit already solved by \cite{HM}
  because Allard's formulation allowed non-stationary varifolds.
  It was also answered by Kiselman \cite{Kiselman}, who constructed a closed positive current in $\mathbb C^2$
  with non-unique tangent cones.
  The current is not rectifiable since its support contains
  separating $3$-dimensional surfaces created by use of $\max$
  in \cite[(4.3)]{Kiselman}, at least when applied as described in Examples 4.2 and 4.3 \cite{Kiselman}.
  He also uses smoothing by convolution.
  Kiselman's example was generalized to general bidegree $(p,p)$ in \cite[Theorem 3.11]{Blel}.
  Also this example is not rectifiable since the current $W=i\partial\bar\partial F$ is added on \cite[p.~528, p.~529]{Blel},
  where $F$ equals a power of $-\Log\left|z\right|^2$ in a neighbourhood of $0$.

 \NewE %

 \medbreak

              The book \cite{Simon} and the paper \cite{Allard1972} are standard sources
              cited when
              varifolds and
              related regularity results
              are of concern.
       Varifolds are generalized (non-oriented) surfaces
       and admit compactness properties suitable to approach
       the problem of existence of surfaces with minimal area.

      On p.~243, L.~Simon recalls the definition of tangent varifolds.
      He proves that if $C$ is a tangent varifold
      (and if some natural conditions are satisfied),
      then $\mu_C$ is conical,
 where
 $\mu_C$
 denotes
 the measure in $\R^n$ associated with $C$ by the direction-forgetting
 projection
 $G_m(
                 \NewBrefsugSmall
         \R^n
                 \NewErefsugSmall
 ) \to \R^n$.
      He says that it seems to be an open question
      whether $C$ itself has to be conical.

 Likewise, W.\,K.~Allard \cite[p.~459--460]{Allard1972}
 states
 that all $C\in \VarTan_a V$ are conical (under some conditions
 \NewBsmall
 on densities of $V$ and $\vari V$)
 \NewEsmall
 and then he says he knows of no varifold
 \NewBsmall
 (with a weak condition on the densities of $V$ and $\vari V$ at $a$)
 \NewEsmall
 such that
 $\VarTan_a V$ has more than one element.
 \NewBsmall
 We already noted that examples of varifolds with properties specified by Allard were provided by \cite{HM}
 (non-stationary, which is not natural in context of \cite{Allard1972})
 and \cite{Kiselman} (non-rectifiable).
 \NewEsmall

  The result that we prove in this paper is the following (see
  Theorem~\ref{thm:nonconical} and Theorem~\ref{thm:conical}).

\begin{theorem}\label{thm:main-v-uvodu}
   There exists a stationary rectifiable $2$-varifold in $\R^4$ that has a non-conical (hence non-unique) tangent at a point.
   There exists a stationary rectifiable $2$-varifold in $\R^4$ that has a conical but non-unique tangent at a point.
   (The varifolds have a positive and finite $k$-dimensional density at the point.)
\end{theorem}
   Note that there is no such varifold $V$ with non-conical tangent
   and $\theta^2(\mu_V,\cdot)$ bounded away from zero
   on $\spt \mu _V$, as the following results imply.
\begin{lemma}\label{l:abouttangents}
    Let $V$ be a stationary $m$-varifold on an open set $\Omega\subset\R^n$,
    $x_0\in \Omega$,
    $C\in \VarTan_{x_0} V$
    and
    $C\neq 0$.\relax
   \footnote
   {Since $C\neq 0$,
    we have
    $\theta^m(\mu_V,x_0) \in (0,\infty)$
    from the Monotonicity formula for stationary varifolds,
    cf.~\cite[40.5]{Simon}.
    Therefore the assumptions of
    Corollary 42.6.
    (namely 42.1.)
    are satisfied.}

   If $\theta^m(C,x)>0$ for $\mu_C$-almost every $x$, then $C$ is conical and rectifiable.
   \textup(Stated on \cite[p.~243]{Simon}, proved in proof of \cite[Corollary 42.6]{Simon}.
       \NewBsmall
   Alternatively see \cite[5.2(2)(b)]{Allard1972} for conicity of $\mu_C$ and
   then
   the rectifiability theorem \cite[5.5(1)]{Allard1972}
   for how this determines the directions $C^{(x)}$ of\/ $C$.\relax
       \NewEsmall
   \textup)

\NewBsmall
   If $C$ is rectifiable, then
   (equivalently)
   $\theta^m(C,x)>0$ for $\mu_C$-almost every $x$
   and hence $C$ is again conical.
\NewEsmall

   If $\theta^m(\mu_V,\cdot)\ge c > 0$
   $\mu_V$-almost everywhere
   then $\theta^m(C,\cdot)\ge c >0$ $\mu_C$-almost everywhere and $C$ is conical
   \textup(and rectifiable\textup)
   \cite[proof of Corollary 42.6]{Simon},
       \NewBsmall
   \cite[6.5]{Allard1972}.
       \NewEsmall
\end{lemma}

 Further note that if $C$ is a tangent varifold from our example, then
 $\mu_C$
 must be
 conical \cite[42.2 on p.~243]{Simon}.

 For stationary $1$-varifolds, the tangent varifolds $C$ are always conical since
 $\mu_C$ is conical and $x \in S$ (equivalently, $p_{S^ \bot} (x) = 0$)
 for all $(x,S) \in \spt C$ \cite[p.~243, l.~2--3]{Simon}.
 \NewBsmall
  For stationary $1$-varifolds with density bounded away from zero,
  the tangent varifolds are conical and unique \cite{AllAlm1976}.
 \NewEsmall

 \NewBsmall
  In Remark on page 449, \cite{Allard1972}
  relates conicity of stationary varifolds to the constancy of its ``sphere slices $B^{(r)}$'' that are implicitly defined by \cite[Theorem~5.2(3)]{Allard1972}.
  Namely, he writes: There is $C$ as in \cite[Theorem 5.2(2), p.~446]{Allard1972}
  (i.e., $C$ a stationary $k$-varifold, with the density $\theta^k(C,0) \ge \mu_C(B_1(\R^n))$)
  which is not homothetically invariant (conical)
  {\em \ if and only if\ }
  there is B as in [Theorem 5.2(3), p.~448]
  (i.e., for almost every $r>0$,
  the slice
  $B^{(r)}$ is a $(k-1)$-varifold in $S_1(\R^n)$,
  which is `stationary in the manifold $S_1(\R^n)$' if $k\ge 2$ resp.\ balanced if $k=1$, with $r\mapsto\mu_{B^{(r)}}$ almost constant and $r\mapsto B^{(r)}$ measurable)
  which is not almost constant.
  The simpler non-rectifiable version of our
  examples
  (see Section~\ref{s:non-rect})
  shows that both statements are (unfortunately) true.
  Note that the example of Kiselman \cite{Kiselman}, and our rectifiable
  example (Sections \ref{s:rect} and \ref{s:variants})
  are
  not applicable to these statements
  (the monotonicity ratio and
  corresponding function $r\mapsto \mu_{B^{(r)}}$ are far from constant).
 \NewEsmall

\section{Notation and definitions}\label{s:xxxxxx}

For $0\le r < s \le \infty$,
denote by $S_r(\R^n)$ the sphere of radius $r$ in $\R^n$
and
$A_r^s=A_r^s(\R^n)=\{x\in \R^n\setcolon  r \le \|x\| \le s \}$ the annulus (or shell) in $\R^n$.
Let $\Sone=S_1(\R^2)$.

\smallbreak

$X$ denotes a smooth compactly supported vector field on $\R^n$ (or on $\Omega\subset \R^n$).

\smallbreak

If $\nu$ is a measure and $M$ is $\nu$-measurable then
$\nu \mrest M$ denotes the restriction of $\nu$ to $M$: $(\nu \mrest M) (A)= \nu (M\cap A)$.

\medbreak

$\phi_{\ImageMeasureFM} \mu$ denotes the image measure \cite[2.1.2]{Federer}:
\begin{equation}\label{eq:image-measure}
              \phi_\ImageMeasureFM \mu (A) = \mu( \phi^{-1} (A) )
       .
\end{equation}
     If $V$ is a $k$-varifold in $\R^n$ (i.e., a measure on $G_k(\R^n)$, see Section~\ref{s:varifolds}),
     then
     we write
     $\phi_\ImageMeasureFV V$
     for
     the
     image measure (if $\dom \phi \subset G_k(\R^n)$)
     defined by \eqref{eq:image-measure}
     and
     \[
     \phi_\ImageVarifold V
     \]
     for
     the
     image varifold (assuming $\dom \phi \subset \R^n$;
     see Section~\ref{s:tan-and-conical}).
     The standard notation for both is the same ($\phi_\# V$)
     which would cause difficulties when reading some expressions in this paper.

\subsection{Varifolds}\label{s:varifolds}
To recall basic notions we
follow and extend \cite[p.~4--5, \S Varifolds]{GMTONeil}.
More details can be found in \cite{Allard1972} and \cite{Simon}.
An {\em $m$-varifold} $V$ on an open subset $\Omega\subset \R^n$
is a Radon measure on
\[
        G_m(\Omega) := \Omega \times G(n,m)
    .
\]
($G(n,m)$ denotes the {\em Grassmann manifold} consisting of $m$-dimensional linear
subspaces of $\R^n$.)
The space of $m$-varifolds is equipped with the {\em weak topology}
given by saying that $V_i \to V$ if and only if
$\int f \intd V \to \int f \intd V$ for all compactly
supported, continuous real-valued functions on $G_m(\Omega)$.
Varifolds can be combined using the addition
which is addition of measures ($ (c_1 V_1 + c_2 V_2) (B) = c_1 V_1(B) + c_2 V_2(B) $).
A countable sum of varifolds is also a varifold, provided it is a Radon measure,
i.e., it assign finite values to compact sets.

To a given $m$-varifold $V$, we associate a Radon measure $ \mu_V $
on $\Omega$ by setting $ \mu _V (A) = V ( G_m(A) )$ for
$A\subset \Omega$.
$\mu_V$ is called the {\em weight} of $V$ (\cite[p.~229]{Simon}).
As a partial converse,
to a (Radon) $m$-rectifiable measure $ \mu $
(see~\cite{GMTONeil})
we can associate an {\em $m$-rectifiable varifold} $V=V_\mu$ by defining
\begin{equation}\label{eq:rectif}
   V(B) =  \mu  \{ x \setcolon (x,T_x) \in B \},
  \qquad
   B\subset  G_m(\Omega)
\end{equation}
where $T_x$ is the approximate tangent plane at $x$.\relax
\footnote{\relax
Although there are several possible definition of approximate tangent plane
(see~\cite{GMTONeil}, \cite[p.~428, (3) and (b)]{Allard1972} and \cite[11.2]{Simon}),
they agree $\mu$-almost everywhere.
The definitions of rectifiable varifolds in \cite{Allard1972} and \cite{GMTONeil}
essentially agree with that of \cite{Simon}, cf.\ footnote on \cite[p.~77]{Simon}.
}
If a countable sum of rectifiable varifolds is also a varifold then it is rectifiable.
In this paper we need only
the following particular case of rectifiable varifolds
(and their countable sums):
$ V = V _ { c \cdot \Hausd^m \mrest S} $
where
$S = \range(U)$
is a smooth parameterized surface
and $c \in (0,\infty)$.
Then the approximate tangent plane $T_{U(x)}$
agrees ($\mu_V$-almost everywhere) with the classical tangent
$ \Span \{ \partial U / \partial x^1, \dots, \partial U / \partial x^m \}$
to $S$,
and $V$ is exactly $c \cdot \mathbf v (S)$ from \cite[p.~431]{Allard1972}.

The support of a measure $\mu$ is denoted by $\spt \mu$.
Note that if $V$ is an $m$-varifold in $\Omega \subset \R^n$
then $\spt V \subset G_m (\R^n)$ while
$\spt \mu_V \subset \R^n$.
If $V$ is an $m$-varifold (hence also a measure)
and we say that $V$ is {\em supported} by a set $M$
if
$M\subset \R^n$
and $\mu_V(\R^n \setminus M) = 0$
or
$M\subset G_m(\R^n)$
and $V(G_m(R^n) \setminus M) = 0$.
                 If $V$ is an $m$-varifold on $\Omega \subset \R^n$ and $M\subset\Omega$ then
                 $V \mrest G_m(M)$ might be called
                 the {\em restriction} of $V$ to $M$.

\medbreak
     The {\em density} that we use in Introduction is defined as
     \[
        \theta^k ( \mu , x) =  \lim_ {r\to 0+} \mu(
        \NewBsmall
        x
        \NewEsmall
        + A_0^r) / r ^ k
     \]
     for a measure $\mu$ on $\R^n$,
     and by $  \theta^k ( V , x) =  \theta^k ( \mu_V , x) $
     for a varifold $V$.

\subsection{The first variation. Stationary varifolds. The mass. The curvature.}\label{s:variation}
The {\em first variation} of an $m$-varifold $V$ is a map from the space of smooth compactly supported vector fields
on $\Omega$ to $\R$ defined by
(see \cite[p.~434]{Allard1972} and \cite[p.~234, p.~51]{Simon})
\begin{equation}\label{eq:var-def}
    \vari V (X) = \int_\Omega  \diver_S X(x) \intd V(x,S)
\end{equation}
where $\diver_S X(x)$ is the divergence at $x$ of the field $X$ restricted (and projected) to affine subspace $x+S$
(\cite[p.~234]{Simon}).
The idea is that the variation measures the rate of change in the 'size' (mass) of the varifold if it is perturbed slightly
(see the alternate formula in \cite[p.~233]{Simon}).
The {\em mass} of the varifold (see \cite[p.~229]{Simon}) is given by
\[
   \mass(V) = V( G_m(\Omega) ) = \mu_V(\Omega)
   .
\]

\medbreak

If $\vari V =0$, then the varifold is said to be {\em stationary}.
Varifold $V_{\Hausd^m \mrest S} $ associated to an $m$-dimensional affine plane $S$
in $\R^n$ is stationary.

\medbreak

Assume
$ V = V _ {\Hausd^m \mrest S} $
is the rectifiable
varifold associated to Hausdorff measure restricted to a smooth surface $S\subset \R^n$
such that the closure $\overline S$ is a $C^2$-smooth compact manifold with
        smooth $(m-1)$-dimensional boundary
      $\partial S := \overline M \setminus M$.
Then \eqref{eq:var-def} reads
\begin{equation}\label{eq:var-smooth}
    \vari V (X) = \int _ S  \diver_{T_x} X(x) \intd \Hausd^m(x)
\end{equation}
and can be (see \cite[7.6]{Simon}) computed as
\begin{equation}\label{eq:var-smooth2}
    \vari V (X)
      =
       - \int _ S    X \cdot \mathbf H  \intd \Hausd^m
       - \int _ {\partial S}  X \cdot \eta \intd \Hausd^{m-1}
\end{equation}
where $\eta$ is the inward pointing unit co-normal of $\partial S$, cf.~\cite[p.~43]{Simon},
and
$\mathbf H$ is the mean curvature vector (\cite[7.4]{Simon}).
If
$U$ is a parameterization of $S$ and
$\mathcal B(x):=\{\partial U/\partial x^1, \dots,
\JKXallowbreak
\partial U/\partial x^m\}$
happens to be orthonormal at $x$
then $\mathbf H$ can be obtained (cf.\ 7.4 together with the last line on p.~44 of \cite{Simon}) as
\begin{equation*}\label{eq:H_unused}
  \mathbf H(U(x)) = \sum_{i=1}^m
              \left(
                     \frac{\partial^2 U(x)}{(\partial x^i)^2}
              \right)
                ^ \bot
\end{equation*}
where $v ^ \bot$ denotes orthogonal projection of $v$ to the orthogonal complement of
$T_{U(x)} = \Span \mathcal B (x)$.
If $\mathcal B(x)$ is merely orthogonal at $x$, a linear change of variables
$ \tilde x_i 
  =
  \sqrt{g^{ii}} x_i = \frac{1}{\|\partial U/\partial x^i\|} x_i$
reveals that
\begin{equation}\label{eq:H}
  \mathbf H(U(x)) = \sum_{i=1}^m
              \left(
                    g^{ii}
                     \frac{\partial^2 U(x)}{(\partial x^i)^2}
              \right)
                ^ \bot
                .
\end{equation}
We skip further derivations
and note for the sake of completeness
that \eqref{eq:H} is in accordance
with  the general
formula
\begin{equation}\label{eq:minimal}
  \mathbf H(U(x))
  =
   \left(
    \smash { \sum_{ij} }
      g^{ij} 
      \frac{\partial^2 U}{\partial x^i \partial x^j}
   \right)
   ^\bot
\end{equation}
where $(g^{ij})$ is the inverse to the metric tensor $(g_{ij})$
of $U$
(see \cite[(1.11), p.~1098]{Oss1969}).

\subsection{An example}\label{s:example}
\begin{exercise}\label{ex:ex}
 Let $H$ be a hyperplane dividing $\R^3$ into two half-spaces $H_1$,
 $H_2$. Let $S_1$, $S_2$ be $2$-dimensional subspaces orthogonal  to
 $H$. For $i=1,2$, let
 $V_i= (\Lebesgue^3\mrest H_i) \times \delta_{S_i}$
 (where $\delta_{S_i}$ is the Dirac measure at the point $S_i\in G(3,2)$)
 and $V=V_1+V_2$.
 Show that $V$ is a stationary varifold on $\R^3$.
\end{exercise}
\begin{solution*}
 From \eqref{eq:var-def} and the divergence theorem
 we have $\vari V_i(X) = - \int_H x\cdot \eta_i \intd \Hausd^2$
 where
 $\eta_i$ is the inward pointing unit normal to $H_i$.
 \qed
\end{solution*}
\begin{interpretation}
    $V_i$ is the integral
    (or,   %
    uncountable ``linear combination'')
    of varifolds $V_{i,x} = V _ {\Hausd^2 \mrest (S_i + x)}$, $x\in S_i ^ \bot$.
    The variations $\vari V_{i,x}$ combine in the same way, and it turns out
    that the result is exactly opposite for $V_1$ and $V_2$.
\end{interpretation}
\begin{remark}
 The varifold $V$
 \NewBsmall
 from Exercise~\ref{ex:ex}
 \NewEsmall
 is a 2-varifold supported by the 3-space
 ($\mu_V =\Lebesgue^3$, $\spt \mu_V = \R^3$);
 $V$ is non-rectifiable.
 $V$ can be ``approximated'' by a rectifiable varifold supported by many half-planes
 touching $H$ and parallel to $S_1$ (inside $H_1$) or $S_2$ (inside $H_2$).
 (The more half-planes, the better approximation and the less density on each of them.)
 This varifold cannot be stationary --- the failure is located near $H$.
 There is a better ``approximation'' that is rectifiable and stationary,
 which is supported by strips of plane creating structure that branches and refines towards $H$.
 \NewBsmall
 See Figure~\ref{fig:network}c) for a planar network of segments illustrating such a branching.
 \NewEsmall
  \qqeedd
\end{remark}
\begin{remark}\label{rem:approximation}
 Also the $2$-varifold
 $
 \tilde V:=
 (\Lebesgue^3\mrest M_1) \times \delta_{S_1}
 +(\Lebesgue^3\mrest M_2) \times \delta_{S_2}$
 is stationary when
 $M_1=\bigcup_{k\in \Z} [2k-1, 2k] \times \R \times \R$,
 $M_2=\bigcup_{k\in \Z} [2k, 2k+1] \times \R \times \R$,
 $S_1=\R\times\R\times\{0\}$,
 $S_2=\R\times\{0\}\times\R$.
 \NewBsmall
 $\tilde V$ can be again approximated by a stationary and rectifiable varifold
 by using the idea from Figure~\ref{fig:network}c)
  twice
 inside
 each $[k, k+1]  \times \R \times \R$.
 \NewEsmall
  \qqeedd
\end{remark}

\subsection{Tangents. Conical varifolds}\label{s:tan-and-conical}
 For $x\in \R^n$ and $\lambda > 0$, let
 \begin{equation}\label{eq:blowup}
     \eta_{x,\lambda} (y) = \frac { y - x }{ \lambda },
     \qquad
     y\in \R^n
     .
 \end{equation}

 If $V$ and $C$ are $m$-varifolds on $\R^n$ and $x\in \R^n$,
 we say that $C$ is a {\em tangent varifold} to $V$ at $x$,
 $C \in \VarTan_x V$,
 if there exist $ \lambda_i>0 $, $\lambda_i \to 0$
 such that, for every
 continuous function $f$ on $G_m(\R^n)$ with compact support,
\[
      \int f(y,S)
                                     \intd C(y,S)
   =
     \lim _ { i \to \infty }
      (\lambda_i) ^ { -m }
      \int f(\eta_{x,\lambda_i}(y),S)
                                     \intd V(y,S)
  .
\]

 This is equivalent to
\[
   \eta_{x, \lambda_i\ \ImageVarifold } V \to C
\]
 (weakly), which is the definition used in \cite[p.~242-243]{Simon}.

 The general definition of ${}_\ImageVarifold$ for varifolds is
 (denoted differently by ${}_\#$)
 in \cite[p.~233]{Simon}
 and it is slightly complicated.
 We need ${}_\ImageVarifold$ only (i) with maps that are combination of translation
 and homothety, like \eqref{eq:blowup}, in which case
\[
     \eta_{x, \lambda_i\ \ImageVarifold } V
     (A)
     =
     (\lambda_i) ^ { -m }
     V( \{ (y,S) \setcolon ( \eta_{x, \lambda_i}(y), S ) \in A \} )
     ;
\]
(ii) with orthonormal linear maps $L$, with
\[
     L_{\ \ImageVarifold } V
     (A)
     =
     V( \{ (y,S) \setcolon ( L(y), L(S) ) \in A \} )
     .
\]

\bigbreak

  An $m$-varifold $C$ is {\em conical} if
\[
     \eta_{0, \lambda\ \ImageVarifold } C = C
\]
  for every $\lambda >0$.

\section{The non-rectifiable varifold}\label{s:non-rect}%

 We start with
 an
 example of a non-rectifiable varifold, which is simpler.
 The rectifiable varifold in later sections is in fact a suitable rectifiable approximation
 of this non-rectifiable example.
 Thus, in this section we prove the following
 weaker version of Theorem~\ref{thm:main-v-uvodu}.
\begin{proposition}\label{prop:nonrectif}
   There is a $2$-varifold in $\R^4$ that has a non-conical (hence non-unique) tangent at a point.
   There is a $2$-varifold in $\R^4$ that has a conical but non-unique tangent at a point.
\end{proposition}

\begin{proof}
 The varifold will be supported by the three-dimensional
 surface,\footnote{\relax
         The surface is neither a linear space nor a convex set: it contains points
         $(1,0,0,0)$ ($a=c=1$, $b=d=0$)
         and
         $(0,0,0,1)$ ($a=c=0$, $b=d=1$)
         but does not contain $(1/2,0,0,1/2)$.
         Indeed, $(t,0,0,t)=(ac,bc,ad,bd)$, $t\neq 0$ leads to
         $a\neq 0$, $c=t/a$, $b\neq 0$, $d=t/b$,
         then $bt/a=0$, $at/b=0$ and finally $b=0=a$, a contradiction.
         }
 in $\R^4$,
 \NewNewBsmall
 for which we propose the name {\em Clifford cone},\relax
 \NewNewEsmall
         \footnote{\relax
         The surface is actually a copy of the three-dimensional cone
         generated by $\Sone\times\Sone$ as can be seen from
         the relation
         $(\cos\gamma,\sin\gamma,\cos\delta,\sin\delta)=(x+w,y-z,x-w,z+y)$
         where $(x,y,z,w)=F((\cos \alpha,\sin\alpha),
         \JKXallowbreak
         (\cos\beta,\sin\beta))$,
         $\gamma=\alpha-\beta$,$\delta=\alpha+\beta$.
         \newline
         The surface was the first known nontrivial minimal cone in $\R^4$,
         \cite[p.~1113]{Oss1969}.
         $\Sone\times\Sone$ is so called Clifford torus.
         Recently, Simon Brendle announced that (up to a congruence)
         it is the only embedded minimal torus in $\mathbb S^3$ \cite{Brendle2012}.
         }
 parameterized by
 \begin{equation}\label{eq:F-def}
  F((a,b),(c,d))=ac\, e_1+ bc\, e_2 + ad\, e_3 + bd\, e_4.
 \end{equation}

 Then, for every $t>0$,
\begin{equation}\label{eq:F-radial}
 F((ta,tb),(c,d))=t F((a,b),(c,d))= F((a,b),(tc,td))
\end{equation}
and
\begin{equation}\label{eq:F-exchange}
  F((a,b),(c,d))
                = c(ae_1 + be_2) + d(ae_3 + be_4)
                = a(ce_1 + de_3) + b(ce_2 + de_4).
\end{equation}

Now, we are ready for an informal explanation of the idea.
The surface is the union of a parameterized family of two-dimensional linear subspaces.
In fact there is a pair of such representations that are ``orthogonal'':
We can fix $(a,b) \in \Sone$ as a parameter and use variables $(c,d) \in \R^2$
to create a $2$-dimensional
varifold
$ V_1^{(a,b)} := V _ {\Hausd^2 \mrest \Span \{ ae_1 + be_2,\ ae_3 + be_4 \} }$
(which is stationary because it is associated to
a 2-plane).
Then
we obtain a new (non-rectifiable) stationary varifold $V_1$ by averaging $V_1^{(a,b)}$
over all $(a,b)\in \Sone$.
We also do the same with swapped $(a,b)$ and $(c,d)$
to obtain a different stationary varifold $V_2$ (yet with $\mu_{V_1} = \mu_{V_2}$).
Suitable parts of the two varifolds can be joined together in similar way as in
Exercise~\ref{ex:ex}, with the
separating hyperplane $H$ replaced by a sphere.
The resulting varifold is again stationary; the quantitative aspects of the formal
proof of this fact depend on the presence of ``orthogonality'' of the parameterizations.
Moreover, we can interleave an infinite number of concentric shells containing
(parts of) $V_1$ and $V_2$ to obtain the target (non-rectifiable) varifold.
Now we proceed with the formal definitions, arguments and calculations.

Let $0\le r < s \le \infty$,
\begin{gather}
  \notag
 g_1((a,b),(c,d))=\Span\{ ae_1 + be_2 , ae_3 + be_4 \}, \\
  \notag
 g_2((a,b),(c,d))=\Span\{ ce_1 + de_3 , ce_2 + de_4 \},
\\
\noalign{\noindent($g_i$ does not depend on all its parameters),}
  \notag
\begin{aligned}
 \phi_{1,r,s} = (F,g_1) \colon S_1(\R^2) \times A_r^s(\R^2) &\to G_2(A_r^s(\R^4)),
 \\
                                 (a,b,c,d) &\mapsto (F((a,b),(c,d)), g_1((a,b),(c,d))),
 \\
  \notag
 \phi_{2,r,s} = (F,g_2) \colon A_r^s(\R^2) \times S_1(\R^2) &\to G_2(A_r^s(\R^4)),
 \\
                                 (a,b,c,d) &\mapsto (F((a,b),(c,d)), g_2((a,b),(c,d))),
\end{aligned}
\\[5pt]
\label{eq:V1def}
  V_{1,r,s} = \phi_{1,r,s}\, _\ImageMeasureFM ( \Hausd^1 \times \Lebesgue^2 ) , \\
\label{eq:V2def}
  V_{2,r,s} = \phi_{2,r,s}\, _\ImageMeasureFM ( \Lebesgue^2 \times \Hausd^1 ) ,
\end{gather}
where
${}_\ImageMeasureFM$ denotes
the image of a measure (the image is a measure that happens to be a varifold),
$\Hausd^1$ is the one-dimensional Hausdorff measure in the unit sphere $S_1(\R^2)$
and $\Lebesgue^2$ is the Lebesgue measure (on the annulus $A_r^s(\R^2) \subset \R^2$).
      From the definition of $m$-varifold we see that $V_{i,r,s}$ defined by
\eqref{eq:V1def}, \eqref{eq:V2def} are $2$-varifolds.
To see that $V_{i,r,s}$ can also be obtained by ``averaging'' (integrating, in the weak sense)
$2$-rectifiable varifolds,
let
\begin{gather}
  \notag
  \phi_{1,r,s,(a,b)}(c,d)=\phi_{1,r,s}((a,b),(c,d)),
  \qquad (c,d) \in A_r^s(\R^2),
  \\
  \notag
  \phi_{2,r,s,(c,d)}(a,b)=\phi_{2,r,s}((a,b),(c,d)),
  \qquad (a,b) \in A_r^s(\R^2),
\\[5pt]
  \label{eq:V1rsab}
  V_{1,r,s,(a,b)} := \phi_{1,r,s,(a,b)}\, _\ImageMeasureFM \Lebesgue^2
  \overset *
  =
  V_ {\Hausd^2 \mrest ( \Span \{ ae_1 + be_2 ,\, ae_3 + be_4 \} \cap A_r^s(\R^4) ) }
  \ ,
  \\
  \label{eq:V2rscd}
  V_{2,r,s,(c,d)} := \phi_{2,r,s,(c,d)}\, _\ImageMeasureFM \Lebesgue^2
  \overset *
  =
  V_ {\Hausd^2 \mrest ( \Span\{ ce_1 + de_3 ,\, ce_2 + de_4 \} \cap A_r^s(\R^4) ) }
  \ ,
\end{gather}
where ``$\overset * =$'' are valid under condition $(a,b) \in \Sone$ or
$(c,d) \in \Sone$, respectively.
Then, by the Fubini theorem,
\begin{gather}
\label{eq:V1eqidef}
  V_{1,r,s} = \int_{(a,b) \in \Sone}   V_{1,r,s,(a,b)}  \intd\Hausd^1, \\
\label{eq:V2eqidef}
  V_{2,r,s} = \int_{(c,d) \in \Sone}   V_{2,r,s,(c,d)}  \intd\Hausd^1.
\end{gather}
Since $V_{i,r,s,(\cdot,\cdot)}$ is just the varifold corresponding
to an annulus part of a $2$-plane
($\mathbf H=0$),
its first variation corresponds to the inward pointing unit co-normal field
supported on the
two circles (cf.~\eqref{eq:var-smooth2}):
\begin{gather*}
  \vari V_{1,r,s,(a,b)}(X) =
    \int\limits_ {\{F(a,b,c,d)\setcolon c^2+d^2=s^2\}} X\cdot N \intd\Hausd^1
  - \int\limits_ {\{F(a,b,c,d)\setcolon c^2+d^2=r^2\}} X\cdot N \intd\Hausd^1,
              \\
  \vari V_{2,r,s,(c,d)}(X) =
     \int\limits_ {\{F(a,b,c,d)\setcolon a^2+b^2=s^2\}} X\cdot N \intd\Hausd^1
  -  \int\limits_ {\{F(a,b,c,d)\setcolon a^2+b^2=r^2\}} X\cdot N \intd\Hausd^1
\end{gather*}
where $N(x)=x/\|x\|$
and where we leave out the first term if $s=\infty$.
The second term is zero if $r=0$.
Integrating over $(a,b)\in S_1(\R^2)$, respective over $(c,d)\in S_1(\R^2)$,
and changing the variables back to the image of $F$ (where it becomes a circle of radius $s$ or $r$),
we get

\begin{gather*}
 \vari V_{1,r,s} (X)=
        \int\limits
         _{F(S_1(\R^2)\times S_s(\R^2))}  X\cdot N
                      \ \ \intd\ \frac{\Hausd^1}{s} \times \Hausd^1
       - \int\limits
         _{F(S_1(\R^2)\times S_r(\R^2))}  X\cdot N
                      \ \ \intd\ \frac{\Hausd^1}{r} \times \Hausd^1
     \\
 \vari V_{2,r,s} (X)=
        \int\limits
         _{F(S_s(\R^2)\times S_1(\R^2))}  X \cdot N
                     \ \ \intd\ \Hausd^1 \times \frac{\Hausd^1}{s}
       - \int\limits
         _{F(S_r(\R^2)\times S_1(\R^2))}  X \cdot N
                     \ \ \intd\ \Hausd^1 \times \frac{\Hausd^1}{r}
\end{gather*}
(Again, if $s=\infty$ or $r=0$, the first or second term has to be replaced by
zero. In particular, $V_{i,0,\infty}$ are stationary.)
Therefore $V_{1,r,s}$ and $V_{2,r,s}$ have the same first variation,
   $\vari V_{1,r,s}  = \vari V_{2,r,s} $,
   and (cf.~\eqref{eq:var-def})
\begin{equation}\label{eq:same-vari}
    \int_\Omega  \diver_S X(x) \intd V_{1,r,s}(x,S)
    =
    \int_\Omega  \diver_S X(x) \intd V_{2,r,s}(x,S)
   .
\end{equation}
We show that
\begin{equation}\label{eq:Vseq}
 V
 =
 V _ {\{r_i\} _ {i\in \Z}}
 =
 \sum_{i=-\infty}^{\infty}
   \left(
            V_{1, r_{2i}, r_{2i+1}}+V_{2, r_{2i+1}, r_{2i+2}}
   \right)
\end{equation}
 is a stationary varifold for any increasing sequence $\{r_i\}_{i\in \Z}$ with
 $\lim_{i\to \infty} r_i = \infty $
 and
 $\lim_{i\to -\infty} r_i = 0 $.

 Indeed,
 V is a Radon measure on $ G_2( \R^4 ) $
 since, e.g., $V( G_2(A_0^s)) = \pi \cdot \pi s^2$.
 Using~\eqref{eq:var-def} and substituting from \eqref{eq:same-vari}
 \begin{multline*}
 \vari V
      (X)
 =
  \sum_{i=-\infty}^{\infty}
    \int_\Omega  \diver_S X(x)
    \intd
            V_{1, r_{2i}, r_{2i+1}}
    +
    \int_\Omega  \diver_S X(x)
    \intd
            V_{2, r_{2i+1}, r_{2i+2}}
 \\
 =
  \sum_{i=-\infty}^{\infty}
    \int_\Omega  \diver_S X(x)
    \intd
            V_{1, r_{i}, r_{i+1}}
 =
    \int_\Omega  \diver_S X(x)
    \intd
            V_{1, 0, \infty}
    =
    0
 \end{multline*}
 since $V_{i,0,\infty}$ is stationary.

 If $r_i=2^i$ then
 $C:=V$ is a non-conical tangent varifold to $V$ at $0\in\R^4$.
 (Also $\eta _{0, \lambda\ \ImageVarifold } V \in \VarTan_0 V$
  for $\lambda \in (0,\infty)$.)

 If $r_i=2^{2^i}$ then
 $V_{1,0,\infty}$ and $V_{2,0,\infty}$
 are two different conical tangent varifolds to $V$ at $0\in\R^4$.
 (Also
 $V_{1,0,r} + V_{2,r,\infty} \in \VarTan_0 V$
 and
 $V_{2,0,r} + V_{1,r,\infty} \in \VarTan_0 V$
 for $r\in (0,\infty)$.)

The above statements about ``non-conical'' tangent and about ``two different'' varifolds
need a bit of justification and
we choose to formulate them separately.
\qed
\end{proof}
\begin{lemma}\label{l:notcon,differ}
Varifold $V = V _ {\{r_i\} _ {i\in \Z}} $ from \eqref{eq:Vseq} is not conical.
Furthermore, $V_{1,r, s} \neq V_{2,r,s}$ for any $0\le r< s \le \infty$.
\end{lemma}
\begin{proof}
We claim that
\begin{equation}\label{eq:ggequal}
    \text{
                if $F(x) = F(y) \neq 0$ then
                $g_1(x) \neq g_2(y)$
          }
      .
\end{equation}
For $x= (x_1,x_2,x_3,x_4) \in \R^4 \setminus\{0\}$,
we have $F^{-1}(x)=\emptyset$ or
\[
F^{-1}(x) \subset
            \{ (\pm t \sqrt{x_1^2+x_3^2},
                \pm t \sqrt{x_2^2+x_4^2},
                \pm \tfrac{\|x\|}{t} \sqrt{x_1^2+x_2^2},
                \pm \tfrac{\|x\|}{t} \sqrt{x_3^2+x_4^2} )\setcolon  t>0\}.
\]
First we show that
\begin{equation}\label{eq:different}
 \text{
     $S_1:=g_1((a,b),\JKXallowbreak (c,d))$ is different
     from $S_2:=g_2((\pm a,\pm b),\JKXallowbreak (\pm c,\pm d))$
     }
\end{equation}
apart from
singular cases $a=b=0$ or $c=d=0$
(when $F((a,b),\JKXallowbreak (c,d))=0$):
Since $g_2$  does not depend on $a$ and $b$,
we have
$S_2 = g_2((a,b),\JKXallowbreak (\pm c,\pm d))$.
Since $g_1$ does not depend on $c$ and $d$,
we can freely change the sign of $c$ (and $d$) in \eqref{eq:different}.
Therefore it is enough to consider
$S_2 = g_2((a,b),\JKXallowbreak (c,d))$.
Assume that $a^2+b^2\neq 0$ and $c^2+d^2\neq 0$.
$S_1$ and $S_2$ are two-dimensional subspaces
and if $S_1=S_2$ then $\Span(S_1\cup S_2)$ is two-dimensional as well,
i.e., the matrix
\[
\left(\begin{smallmatrix}
a & b & 0 & 0 \\
0 & 0 & a & b \\
c & 0 & d & 0 \\
0 & c & 0 & d
\end{smallmatrix}\right)
\]
has rank $2$.
Then
$
(a^2+b^2)c=
-
\left|\begin{smallmatrix}
a & b & 0 \\
0 & 0 & a \\
0 & c & 0
\end{smallmatrix}\right|
+
\left|\begin{smallmatrix}
a & b & 0 \\
0 & 0 & b \\
c & 0 & 0
\end{smallmatrix}\right|
= 0 + 0 = 0
$,
$
(a^2+b^2)d=
\left|\begin{smallmatrix}
a & 0 & 0 \\
0 & a & b \\
0 & 0 & d
\end{smallmatrix}\right|
-
\left|\begin{smallmatrix}
b & 0 & 0 \\
0 & a & b \\
0 & d & 0
\end{smallmatrix}\right|
= 0
$.
Hence $c=0$, $d=0$, a contradiction showing that \eqref{eq:different} is true.

Since $g_i((ta,tb),(uc,ud))=g_i((a,b),(c,d))=S_i$ for $i=1,2$ and
$t,u\in\R\setminus\{0\}$,
we get \eqref{eq:ggequal}.

By \eqref{eq:ggequal},
$V_{1,r,s}$ and $V_{2,r,s}$ are supported by disjoint subsets of $G_2(\R^4)$
whenever $r>0$.
(For $r=0$, $\spt V_{1,r,s} \cap \spt V_{2,r,s} \subset \{ (0,0,0,0) \} \times G(4,2)$.)
Hence
$V_{1,r, s} \neq V_{2,r,s}$ for any $0\le r< s \le \infty$.
Obviously,
varifold
$V = V _ {\{r_i\} _ {i\in \Z}} $ is not conical.
\qed
\end{proof}

\section{The rectifiable varifold}\label{s:rect}%
 The rectifiable stationary varifold
 \NewBsmall
 (let us call it $V_{\text{rect}}$ for now)
 \NewEsmall
 will be obtained as a suitable approximation of
 the above non-rectifiable $V$.
 \NewBsmall
 (The idea of approximation for the case of linear configuration was suggested in Remark~\ref{rem:approximation},
 cf.\ Figure~\ref{fig:network}c).
 Since we work in central configuration,
 our task will be more complicated,
 see the caption of Figure~\ref{fig:network}b) where the simplest idea does not transfer from Figure~\ref{fig:network}a).)
 \NewEsmall

 Instead of planar
 \NewBsmall
 annuli
 (see the support of $V_{1,r,s,(\cos t_1,\sin t_1)}$ in \eqref{eq:V1rsab})
 \NewEsmall
 smoothed out by averaging in the definition of $V$,
 the support now consists of
 \NewBsmall
 countably many
 planar annuli
 and
 countably many
 \NewEsmall
 pieces
 \NewBrefsugSmall
 homeomorphic to annuli (we will call them
 ``rings'')
 whose number will increase
 (through a process that we call ``branching'')
 \NewErefsugSmall
 towards the boundary of the layer.

 \NewBsmall
      Since now the pieces of $
      V_{\text{rect}}
      $ are not oriented radially
      ($x \notin S$ for many $(x,S)\in \spt
      V_{\text{rect}}
      $), the ratio
      $ V_{\text{rect}}(
      G_2(B(0,r)))/r^2$ necessarily decreases as $r$ decreases.
      (This is a corollary to the Monotonicity formula
      \cite[17.5]{Simon}.)
      Therefore we have, and do, take special care
      to make sure that the density
      $\theta^2(
      V_{\text{rect}}
      , 0)$ does not vanish.
 \NewEsmall

 \medbreak

 The proof
 continues towards the end of this paper
 and
 depends on the calculations summarized in the following lemmata.

 \medbreak

 The varifold will be again supported by the three-dimensional surface
 parameterized by $F$, see~\eqref{eq:F-def} and \eqref{eq:F-radial}.

 In every point of $x\in \range F \setminus \{ 0 \}$, we will frequently refer
 to the radial direction $N(x) = x / \left\| x \right\|$ and to a selected
 tangential direction. The latter is conveniently expressed by matrix multiplication.

      Let $J_{13}^{24}$ be the matrix that rotates $ e_1 \to e_3$ and $e_2 \to e_4$ given by
\begin{equation}\label{eq:J01}
   J_{13}^{24} =
   J(0,1)
           = \begin{pmatrix}
                      0 & 0 & -1 & 0 \\
                      0 & 0 & 0 & -1 \\
                      1 & 0 & 0 & 0 \\
                      0 & 1 & 0 & 0
            \end{pmatrix}
            .
\end{equation}
      For $\varepsilon \ge 0$
      and
      $x\in \R^4 \setminus\{0\}$,
 \NewBrefsugSmall
      consider the following set of $2$-dimensional planes in $\R^4$
 \NewErefsugSmall
\begin{multline}
\label{eq:Gradx}
       G_{\RadAndOneThree}^{\,\varepsilon}(x)
       =
       \{
          \Span \{u,v\} \setcolon
                       \{u,v\}\text{ orthonormal},\
       \\
                       \| u - N(x) \| \le \varepsilon,\
                       \| v - J_{13}^{24} N(x) \| \le \varepsilon
       \},
\end{multline}
        and
 \NewBrefsugSmall
        related subset of $\R^4 \times G(4,2)$
 \NewErefsugSmall
\begin{equation}
\label{eq:Grad}
       G_{\RadAndOneThree}^{\,\varepsilon}
       =
       \{
          (x, S)\setcolon
                       x \in \R^4 \setminus\{0\},\
                       S \in G_{\RadAndOneThree}^{\,\varepsilon}(x)
       \}.
\end{equation}
        Then $ \{ G_{\RadAndOneThree}^{\,\varepsilon} (x) \setcolon \varepsilon > 0\}$
        is a neighbourhood base for a special point
        $\Span \{ N(x) ,
        \JKXallowbreak
        J_{13}^{24} N(x) \} \in G(4,2)$,
        which is the span of the radial direction and the direction determined by $\JOneThree$.
        From this comes the subscript in our notation.

        Note that if we let
\begin{equation}\label{eq:J1234}
       J_{12}^{34} = \begin{pmatrix}
                      0 & -1 & 0 & 0 \\
                      1 & 0 & 0 & 0 \\
                      0 & 0 & 0 & -1 \\
                      0 & 0 & 1 & 0
            \end{pmatrix}
\end{equation}
        and define
        $G_{\RadAndOneTwo}^{\,\varepsilon}(x)$ accordingly
        then there is $\varepsilon_0>0$ (independent of $x$) such that
\begin{equation}\label{eq:J1324J1234}
         G_{\RadAndOneThree}^{\,\varepsilon}(x)
         \cap
         G_{\RadAndOneTwo}^{\,\varepsilon}(x)
         =
         \emptyset
\end{equation}
        for all $\varepsilon \in [0, \varepsilon_0)$.
        To see that, we first observe that
\begin{equation}\label{eq:differentSpan}
        \Span \{ N(x), J_{13}^{24} N(x) \}
        \neq
        \Span \{ N(x), J_{12}^{34} N(x) \}
        .
\end{equation}
        This is similar to \eqref{eq:different} but now the proof is even easier.
        First consider \eqref{eq:differentSpan} in the special case when $N(x)= e_1$.
        Then if \eqref{eq:differentSpan} were not valid, then the matrix
        \[
           \begin{pmatrix}
              1 &  0 &  0 &  0  \\
              0 &  1 &  0 &  0  \\
              0 &  0 &  1 &  0
           \end{pmatrix}
        \]
        would have range at most two.
        Using a rotation, we see that \eqref{eq:differentSpan} in general is
        equivalent to \eqref{eq:differentSpan} in the special case when $N(x)= e_1$.
 \NewBsmall
        Moreover, we see that $\varepsilon_0>0$ admissible for \eqref{eq:J1324J1234}
        is independent of $x\in \R^4\setminus \{0\}$.
 \NewEsmall

\subsection{Basic surface, rings and their joins.}
 \NewB
    Let $d > 0$, $\alpha_0 \in \R$.

     We will use pieces of a minimal surface that is derived by a ``rotation'' in $\R^4$ from
     planar curve (in polar coordinates)
\begin{equation}\label{eq:ralpha}
       r(\alpha) =
       r^{d,\alpha_0}(\alpha) =
       \sqrt { d \, / \cos 2 ( \alpha - \alpha_0 )},
    \qquad
    \alpha - \alpha_0 \in (-\pi/4, \pi/4)
    .
\end{equation}
     It turns out that the curve is a hyperbola.
     To make a geometrical picture, let us consider the special case $\alpha_0 = \pi/4$
     (the general case is obtained by rotations in the plane);
     then $r(\alpha)^2 = d \, / 2 \sin \alpha \cos \alpha$
     and
     \[
     \{ (r(\alpha)\cos \alpha, r(\alpha)\sin \alpha) \}
     =
     \{ (x_1, x_2) \subset (0,\infty)^2 \setcolon 2x_1 x_2 = d \}
     .
     \]
     The most important are the portions of the hyperbola far away from the origin,
     that is with $x_1$ (respectively, $x_2$) restricted to interval
     $(\varepsilon_1, \varepsilon_2)$ close to $0$.
     This corresponds to $\alpha - \alpha_0 \in (t_1,t_2)$ close to either $-\pi/4$
     or $\pi/4$.

     The rotation applied
     is the one obtained from $F$ (see \eqref{eq:F-def}):
  \begin{equation}\label{eq:UfromF}
     U(\alpha,\beta)
   =
     F(
     (r(\alpha) \cos \alpha, r(\alpha) \sin \alpha)
     , (\cos \beta, \sin \beta))
     ,
  \end{equation}
     where $\alpha - \alpha_0 \in (t_1,t_2) \subset (-\pi/4, \pi/4)$  and $\beta \in \R$.
     The set thus obtained in \eqref{eq:Stt} below is homeomorphic
     (and nearly isometric, for suitable pairs $t_1,t_2$)
     to a planar annulus and we call it a ``ring''.

     For imagination of the full surface, one might notice that it is obtained by
     deformation of a part (a strip, since $\alpha$ is restricted to an interval)
     of Clifford torus
     $F( S_1(\R^2)\times S_1(\R^2))$,
     with the middle circle ($\alpha=\alpha_0$) scaled at $\sqrt{d}$
     and the ends ($\alpha \to \alpha_0 \pm \pi/4$) lifted in the radial direction to infinity.
     As we already indicated above, we will use only rather flat parts
     that are
     lifted high above $\sqrt{d}$
     \NewNewBsmall
     and have nearly radial directions.
     \NewNewEsmall

    \medbreak

     The notation.
     The lemma below summarizes the properties of our minimal surface $S=\range U$.
     The letter $S$ with various indexes denotes parts of the surface
     while $V$ is used for the corresponding varifolds.
     The upper index is reserved for the parameters $d$ and $\alpha_0$;
     we drop them from
     $
     S_{t_1,t_2}
     =
     S^{d,\alpha_0}_{t_1,t_2}
     $
     but we have to leave them in
     $
     V^{d,\alpha_0}_{t_1,t_2}
     $
     (see \eqref{eq:ring-def} below).
     Lower index denotes the range for the variable $\alpha$. The range is either an interval $(t_1, t_2)$
     (thus $S_{t_1,t_2}$ is our ``ring'')
     or a single point $t_1$; thus $S_{t_1}$ is a circle. Boundary of the ring
     $S_{t_1,t_2}$ consists of two circles $S_{t_1}$, $S_{t_2}$.
     Some time later the circle will be denoted by $K(\varrho,\alpha)$
     ($\varrho$ will be the radius) ---
     this change in the notation will be necessary as to drop the dependence on 
     $d$ and $\alpha_0$. Thus $S_{t_1}=K(r^{d,\alpha_0} (t_1) , t_1)$.

  \begin{remark}
  \NewNewBsmall
   $S=\range U$ and all other objects until
   Lemma~\ref{l:system-A}
   are included in
   the Clifford cone
   \ \
   $\range F
   \JKXallowbreak
   =F(\R^2\times \Sone) = \R \cdot F(\Sone\times\Sone)$.
   Their geometry is determined in the $(r,\alpha)$ plane
   and all of them are invariant with respect to the parameter $\beta$ in \eqref{eq:UfromF}
   and the corresponding rotation.
   The rotation introduces factor $r$ into the area functional and
   influences
   thus the shape of the minimal surface.
   Therefore we stick with the full $4$-dimensional description
   and do not restrict ourselves to the $(r,\alpha)$ plane where
   everything is determined.
   In Lemma~\ref{l:system-B}, the roles of $\alpha$ and $\beta$ are interchanged.
  \NewNewEsmall
  \end{remark}

 \NewE

\begin{lemma}\label{l:basic}
 1.
    Let $d > 0$,
    $\alpha_0 \in \R$
    and
    $r(\alpha)$ as in \eqref{eq:ralpha}.
    Consider the parameterized surface
    $
     U (\alpha  , \beta)
     =
       U^{d,\alpha_0} (\alpha  , \beta)
    $,
\begin{multline*}
     U(\alpha,\beta)
     =
      (
      r(\alpha) \cos \alpha \cos \beta,
      r(\alpha) \sin \alpha \cos \beta,
      r(\alpha) \cos \alpha \sin \beta,
      r(\alpha) \sin \alpha \sin \beta
      )
      ,
   \\
     \alpha - \alpha_0 \in (-\pi/4, \pi/4),\ \beta \in \R
   .
\end{multline*}
   \textup($U$ is $2\pi$-periodic in $\beta$, and injective on every period.\textup)
   Then $U$ is a minimal surface.

   2.
   Let
\begin{align}
  \label{eq:Stt}
     S_{t_1,t_2}
     &:=
     \{
       U(\alpha,\beta)
      \setcolon
        \alpha \in (t_1, t_2),\ \beta \in \R
      \}
      ,
   \qquad\qquad\text{\textup(the ``ring''\textup)}
  \\
  \label{eq:St}
     S_{t_1}
     &:=
     \{
       U(t_1,\beta)
      \setcolon
         \beta \in \R
      \}
      ,
  \\
  \notag
   S &:=
      \range( U )
     .
\end{align}

   Then the rectifiable varifold
   $V_{\Hausd^2 \mrest S}$
   is stationary.
   \label{A:tady-is-stationary}

   3.
   \textup(The ring varifold and its first variation.\textup)
   For every $x\in S$,
   find any $p$ satisfying $U(p) = x$
   and let
      \[
           \pmb \eta  _{\alpha_0} (x)
                          =
                       N(   \partial U( p ) / \partial \alpha )
      \]
   where $N(y) = y / \| y \|$.

   For  $ \alpha_0 - \pi/4 < t_1 < t_2 < \alpha_0 + \pi/4$,
   let
   \begin{equation}\label{eq:ring-def}
    V^{d,\alpha_0}_{t_1,t_2}
    =
    V_{\Hausd^2 \mrest S_{t_1,t_2}}
    \ .
   \end{equation}
   Then,
   \begin{equation}\label{eq:ring-vari}
    \vari
    V^{d,\alpha_0}_{t_1,t_2}
    (X)
    =
         \int _ {S_{t_2}}  X(x) \cdot
                           \pmb \eta _{\alpha_0}(x)
                                       \intd \Hausd^{1}
       - \int _ {S_{t_1}}  X(x) \cdot
                           \pmb \eta _{\alpha_0}(x)
                                       \intd \Hausd^{1}
       .
   \end{equation}

   4.
   \textup(Two rings at touch.\textup)
   If
   $\alpha_1 \le \alpha \le \alpha_2$
   and
   $\alpha - \alpha_1 = \alpha_2 - \alpha \in [0, \pi/4)$ then
   \begin{equation}\label{eq:UU}
    U^{d, \alpha_1} (\alpha, \beta) = U^{d,\alpha_2} (\alpha, \beta)
   \end{equation}
   and
   \begin{equation}\label{eq:two-rings}
     \pmb\eta_{\alpha_1} ( U^{d,\alpha_1} (\alpha, \beta) )
     -
     \pmb\eta_{\alpha_2} ( U^{d,\alpha_2} (\alpha, \beta) )
     =
           2 \sin 2(\alpha-\alpha_1)
           \cdot N( U^{d, \alpha_1} (\alpha, \beta) )
   \end{equation}
   is a radial vector at the point.

   5. The tangent plane to $U=U^{d,\alpha_0}$ at $x=U(\alpha, \beta)$ belongs to
      $
       G_{\RadAndOneThree}^{\, 2 \cos 2(\alpha-\alpha_0) }(x)
      $ and
   \[
      \spt
      V^{d,\alpha_0}_{t_1,t_2}
      \subset
              G_{\RadAndOneThree}^{\,\varepsilon}
   \]
   where $\varepsilon = 2 \max \cos([ 2(t_1-\alpha_0), 2(t_2-\alpha_0)])$.

   6. \textup(Mass distribution\textup)
      \[
      \mass (V^{d,\alpha_0}_{t_1,t_2})
        =
          \pi d ( \tan 2(t_2-\alpha_0) - \tan 2(t_1-\alpha_0) )
          .
      \]
      For every $0< \sqrt d \le  r_1 < r_2$ there is a number
      $ \rho =\rho(d,r_1,r_2) \in [r_1, r_2] $
      such that whenever
$\alpha_0 < t_1 < t_2 < \alpha_0 + \pi/4 $,
      and $t_1 \le s_1 \le s_2 \le t_2$
      then,
\begin{multline*}
      \mass (V^{d,\alpha_0}_{t_1,t_2})
        =
          \pi \left| \sqrt { r(t_2)^4 - d^2 } - \sqrt { r(t_1)^4 - d^2 } \right|
\\
        =
           \tfrac{1}{\sqrt{1-d^2/(\rho(d,r(t_1), r(t_2)))^4}}
           \,
           \pi \left| r(t_2)^2 - r(t_1)^2 \right|
\end{multline*}
      and
\begin{multline*}
           V^{d,\alpha_0}_{t_1,t_2} ( G_2 ( A _ {r(s_1)} ^ {r(s_2)} ) )
        =
      \mass (V^{d,\alpha_0}_{s_1,s_2})
        =
          \pi \left| \sqrt { r(s_2)^4 - d^2 } - \sqrt { r(s_1)^4 - d^2 } \right|
\\
        =
           \tfrac{1}{\sqrt{1-d^2/(\rho(d,r(s_1),r(s_2))^4}}
           \,
           \pi \left| r(s_2)^2 - r(s_1)^2 \right|
  .
\end{multline*}
      If $\alpha_0 - \pi /4 < t_1 < t_2 < \alpha_0$, the same holds with
      $A _ {r(s_1)} ^ {r(s_2)}$
      replaced by
      $A _ {r(s_2)} ^ {r(s_1)}$
      and $\rho$ extended by formula
      $\rho(d,r_1,r_2):=\rho(d,r_2,r_1)$ for $\sqrt d\le r_2< r_1$.
\end{lemma}

  \begin{remark}
  For $\alpha_0 + \pi/4 - \varepsilon < t_1<t_2< \alpha_0 + \pi/4$
  (or analogously for $\alpha_0 -\pi/4 < t_1 < t_2 < \alpha_0 - \pi/4 + \varepsilon $),
  and $r=r(t_1)$, $s=r(t_2)$,
  the ring $S_{t_1,t_2} $ is intended to be a perturbation of the annulus supporting $V_{1,r,s,(\cos t_1,\sin t_1)}$
  from \eqref{eq:V1rsab}.
  \qqeedd
  \end{remark}

\NewB
  \begin{remark}
   The surface $S=\range(U)$ can be found in \cite{Lawlor}.
  \qqeedd
  \end{remark}
\NewE

  We will give two arguments for the minimality of surface $U$, the first one is
  easy but slightly incomplete:
  Let $ \alpha_0 -\pi/4 < t_1 < t_2 <  \alpha_0 + \pi/4$
  with $t_2$ close to $t_1$,
  and
  consider the part
  of the surface determined by the range $t \in (t_1, t_2)$
  (cf.~\eqref{eq:Stt});
  recall
  this is the surface created
  by a certain ``rotation''
  from curve
  \[
  \gamma(t) := ( r(t) \cos t, r(t) \sin t, 0, 0),
  \qquad
  t\in (t_1, t_2)
  .
  \]
  The boundary of the selected part consists of two circles
  $S_{t_1}$,
  $S_{t_2}$
  (see~\eqref{eq:St}).
  To this correspond fixed values
  $\gamma(t_1)$,
  $\gamma(t_2)$,
  as boundary conditions for $\gamma$.

  Our first and incomplete argument for the minimality of $U$ is based on comparing
  the area of the selected part of $U$ with surfaces corresponding
  to other possible curves $\gamma$ in $\R^2 \times \{0\}^2$
  with the same boundary condition.

  The area is given by the formula
  \[
     A = 2\pi \int _ {(t_1, t_2)} \| \gamma'(t) \| \cdot \| \gamma(t) \| \intd t
  \]
  since the length of the circle through $\gamma(t)$ is $2\pi \| \gamma(t) \|$.
  We will view $\gamma$ as a curve in $\R^2 \cong \R^2 \times \{0\}^2$,
  and assume that $\gamma$ is the graph of a function $r$
  in polar coordinates,
  that is
  $\gamma(t) = ( r(t) \cos t, r(t) \sin t) $.
  On $\R^2 = \mathbb C$, consider the map $z\mapsto z^2$ whose derivative is $2 z$.
  That maps curve $\gamma$ to a curve $\gamma^2$
  (where $\gamma^2(t) = (\gamma(t))^2 \in \mathbb C$)
  whose length
  \[
       L = \int _ {(t_1, t_2)} \| (\gamma^2)'(t) \| \intd t
         =  \int _ {(t_1, t_2)} 2 \| \gamma'(t) \| \cdot \| \gamma(t) \|  \intd t
  \]
  we find to be directly proportional to $A$.
  It is well known that $L$ is minimal if $\gamma^2$ is the segment connecting its endpoints.
  A special case is a vertical segment given in polar coordinates by
  $(\tilde r, \tilde \alpha)$ with $ \tilde r = d \, / \cos \tilde \alpha $;
  the general case is $ \tilde r = d \, / \cos (\tilde \alpha - \tilde \alpha_0) $.
  Since $z\mapsto z^2$ is expressed in polar coordinates as
  $(r, \alpha) \mapsto (\tilde r, \tilde \alpha) = (r^2, 2 \alpha)$,
  we obtain the curve
  $\gamma(t) = ( r(t) \cos t, r(t) \sin t) $
  with
  $r(t) = \sqrt { d \, / \cos 2(t - \alpha_0) }$, $t\in [t_1,t_2]$.
  The corresponding rotation surface is our best candidate for
  the minimum area surface spanned between $S_{t_1}$ and $S_{t_2}$
  and $U$ likely is a minimal surface.
\begin{theopargself}
\begin{proof}[of Lemma \ref{l:basic}]
  \noindent\newline
  1.
  For formal verification of the minimality of surface $U$,
  it is enough to verify
  that
  $\mathbf H (U) = 0$.

  For $a,b,\alpha,\beta \in\R$, let
\[
 B = B(\beta) = J(\cos \beta, \sin \beta),
 \qquad
 \text{ where }
 \qquad
 J(a,b) 
 =
 \begin{pmatrix}
  a  &  0  & -b  &  0 \\
  0  &  a  &  0  & -b \\
  b  &  0  &  a  &  0 \\
  0  &  b  &  0  &  a
 \end{pmatrix}
\]
and (since we choose to treat the vectors, including $U$, as column vectors, we will
distinguish that in notation from this moment)
\[
 A = A(\alpha) = (\cos \alpha, \sin \alpha, 0, 0)^T.
\]
Then
\[
 U = r B A
\]
where $r$ is a function of $\alpha$:
\[
 U(\alpha, \beta) = r(\alpha) B(\beta) A(\alpha), 
 \qquad
  \alpha \in (-\pi/4, \pi/4),\ \beta \in \R
  .
\]
Note that obviously $\|U\| = r$, hence
\[
 N( U ) = B A
 .
\]
We have
\begin{align}
 \label{eq:Uder-alpha}
 \frac {\partial U}{\partial \alpha} &= r' B A + r B A' = B ( r' A + r A' )
 \\
 \label{eq:Uder-beta}
 \frac {\partial U}{\partial \beta} &= r B' A
\end{align}
where
\[
 A' = ( -\sin \alpha, \cos \alpha, 0, 0)^T,
 \qquad
 B' = J( -\sin \beta, \cos \beta)
 .
\]
Furthermore,
\begin{equation*}\label{eq:druha-unused}
 A'' = ( -\cos \alpha, -\sin \alpha, 0, 0)^T= -A,
 \qquad
 B'' = J( -\cos \beta, -\sin \beta) = -B
\end{equation*}
and hence
\begin{align}
 \label{eq:druhaA}
 \frac {\partial^2 U}{\partial \alpha^2} &=  B ( r'' A       + 2 r' A' + r A'' )
                                          =  B ( (r'' - r) A + 2 r' A' )
 \\
 \label{eq:druhaB}
 \frac {\partial^2 U}{\partial \beta^2} &= r B'' A = - r B A
\end{align}

Obviously
\begin{equation}\label{eq:AA'ortonormal}
   A^T A = (A')^T A' = 1
   \qquad
   A^T A' = (A')^T A = 0
   .
\end{equation}
It is immediate that
$ J(a,b)^T = J(a,-b)$ and $J(a,b)^T J(a,b) = (a^2 + b^2) I$
where $I$ is the identity matrix; in particular
\begin{align}
\label{eq:BtB=I}
   B^T B &= I,
   \\
\label{eq:B'tB'=I}
   (B')^T B' &= I
   .
\end{align}
Hence
\begin{equation}\label{eq:Binv}
   B^{-1} = B^T
   .
\end{equation}
Furthermore, $ J(b, a) J(a, b) = J(0 , a^2 + b^2) $, in particular
\begin{equation}\label{eq:BtB'=J}
   B^T B' = J(0,1)
\end{equation}
and $ B' B^T  = J(0,1)$. Multiplying that by $B$ from the right (see~\eqref{eq:Binv}) we
get
\begin{equation}\label{eq:B'}
    B' = J(0,1) B
    .
\end{equation}

The metric tensor is
\begin{align}
\label{eq:g11}
  g_{11}
  = 
  \frac {\partial U}{\partial \alpha}
  \cdot
  \frac {\partial U}{\partial \alpha}
  &=
  ( r' A + r A' )^T B^T B ( r' A + r A' ) 
  \\
 \notag
  &\overset{\eqref{eq:BtB=I}}{=}
  ( r' A + r A' )^T ( r' A + r A' )
  \overset{\eqref{eq:AA'ortonormal}}{=}
  (r')^2 + r^2
\\
 \notag
  g_{22}
  = 
  \frac {\partial U}{\partial \beta}
  \cdot
  \frac {\partial U}{\partial \beta}
  &=
  r A^T (B')^T r B' A
  \\
 \notag
  &\overset{\eqref{eq:B'tB'=I}}{=}
  r^2 A^T A
  \overset{\eqref{eq:AA'ortonormal}}{=}
  r^2
\\
\label{eq:g12}
  g_{12} = g_{21}
  = 
  \frac {\partial U}{\partial \alpha}
  \cdot
  \frac {\partial U}{\partial \beta}
  &=
  ( r' A + r A' )^T B^T r B' A
  \\
 \notag
  &\overset{\eqref{eq:BtB'=J}}{=}
  r ( r' A + r A' )^T J(0,1) A
  =
  0
\end{align}
 since $  A, A'   \in \R^2 \times \{0\}^2$
 while $ J(0,1) A \in\{0\}^2 \times \R^2$.
 Therefore
\begin{equation*}
   (g_{ij}) = \begin{pmatrix} (r')^2 + r^2 & 0 \\ 0 & r^2 \end{pmatrix},
   \qquad
   (g^{ij}) = \begin{pmatrix} \tfrac{1}{(r')^2 + r^2} & 0 \\ 0 & \tfrac{1}{r^2} \end{pmatrix}.
\end{equation*}
We want to verify $\mathbf H(U)=0$ using \eqref{eq:H} (or, equivalently, \eqref{eq:minimal}).
Thus we want to verify
\[
   v^{\bot} = 0,
   \qquad
   \text{that is,}
   \qquad
   v \in \Span \left\{ \tfrac{\partial U}{\partial x^i} \right\}
\]
where
\[
 v
   =
   \tfrac{1}{(r')^2 + r^2}
   B ( (r'' - r) A + 2 r' A' )
   + \tfrac{1}{r^2}
   ( - r B A )
   .
\]
That is
\[
   \tfrac{1}{(r')^2 + r^2}
   B ( (r'' - r) A + 2 r' A' )
   + \tfrac{1}{r^2}
   ( - r B A )
   \in
   \Span \{
        B ( r' A + r A' ) ,
        r B' A
   \}
   .
\]
Multiplying  by $B^{-1}$ and using \eqref{eq:Binv}, \eqref{eq:BtB'=J}, we get equivalent
relation
\[
   \tfrac{1}{(r')^2 + r^2}
   ( (r'' - r) A + 2 r' A' )
   - \tfrac{1}{r} A 
   \in
   \Span \{
         r' A + r A'  ,
        r J(0,1) A
   \}
   .
\]
Since $A, A' \in \R^2 \times \{ 0 \} ^2$, while $r J(0,1) A \in \{ 0 \} \times \R^2$, 
the latter can be removed:
\[
   \tfrac{1}{(r')^2 + r^2}
   ( (r'' - r) A + 2 r' A' )
   - \tfrac{1}{r} A 
   \in
   \Span \{
         r' A + r A'
   \}
   .
\]
Now the relation reduces to $\R^2 \times \{ 0 \} ^2$, where $A, A'$ form an orthogonal base.
We have $ r' A + r A' \mathbin{\bot} r A - r' A'$ and our relation is equivalent to
\[
   ( r A - r' A' ) ^T
   \left(
   \tfrac{1}{(r')^2 + r^2}
   ( (r'' - r) A + 2 r' A' )
   - \tfrac{1}{r} A 
   \right )
   =
   0
   .
\]
Using \eqref{eq:AA'ortonormal} this reduces to
\[
  r r''  - 3 (r')^2 - 2 r^2  = 0
  .
\]
It is easy to check that our function $r(\alpha) = \sqrt{ d \, / \cos 2 (\alpha - \alpha_0)}$
verifies this equation.

Thus
we proved
that
the mean curvature vector $\mathbf H(U)$ is identically zero
and
$U(\alpha, \beta)$ is a minimal surface.

2.
Since $\mathbf H(U) = 0$ and there is no boundary
($U$ is defined on $\R^2$ and essentially injective)
the associated varifold is stationary.

3.
To obtain \eqref{eq:ring-vari}, it is enough to use
\eqref{eq:var-smooth2};
The boundary of
$S_{t_1,t_2}$
is $S_{t_1} \cup S_{t_2}$,
and
if $U(p) \in \partial S_{t_1,t_2}$ then
$\partial U(p)/\partial \beta$ is obviously tangent to
$\partial S_{t_1,t_2}$
and
$\eta:=\partial U(p)/\partial \alpha$ is orthogonal to it, see~\eqref{eq:g12}.
If $p=(t_1,\beta)$ then $\eta$ is an inner normal, if $p=(t_2, \beta)$ then it is outer.

4.
    Assume now that
    $\alpha _1 \le \alpha \le \alpha _2$ and
    \begin{equation}\label{eq:alpha1alpha2}
    \alpha - \alpha_1 = \alpha_2 - \alpha \in [0, \pi/4)
    .
    \end{equation}
    Then
    $r^{d,\alpha_1}(\alpha) =
     r^{d,\alpha_2}(\alpha)$
    and
    hence
    $ U^{\alpha_0} (\alpha, \beta) = U^{\alpha_1} (\alpha, \beta)$.

   At any point $(\alpha,\beta)$ satisfying \eqref{eq:alpha1alpha2}
   we have, by \eqref{eq:Uder-alpha} and \eqref{eq:g11},
\begin{align}
\notag
 \frac {\partial U}{\partial \alpha} &= r' B A + r B A'
\\
\label{eq:tangent-rad}
     N\left( \frac {\partial U}{\partial \alpha}
     \right)
   &=
   \frac{ r'}{\sqrt{ r'^2 + r^2 }} B A + \frac{r}{\sqrt{ r'^2 + r^2 }} B A'
\end{align}
where $A$, $B$ and $r$ are the same regardless
if
$ U^{d, \alpha_1}$ or $U^{d, \alpha_2}$ is considered.
Only $r'$ is different:
\[
   (r^{d,\alpha_1})'(\alpha)
   = - (r^{d,\alpha_2})'(\alpha)
   .
\]

Letting, e.g., $\alpha_0:=\alpha_1$, we have
\begin{align}
\label{eq:r-formula}
      r &= \sqrt d \, \cos ^{-1/2} 2(\alpha-\alpha_0)
      \\
\label{eq:rbar-formula}
      r' &=  \sqrt d \, \cos ^{-3/2} 2(\alpha-\alpha_0) \sin 2(\alpha-\alpha_0)
      \\
\label{eq:sqrt-formula}
      \sqrt{ r'^2 + r^2 } &=
            \sqrt d \, \cos ^{-3/2} 2(\alpha-\alpha_0)
     .
\end{align}
Since \eqref{eq:tangent-rad} are the values of
$\pmb \eta _{\alpha_1}$ and $\pmb \eta _{\alpha_2}$,
we get \eqref{eq:two-rings}, that is,
\[
     \pmb\eta_{\alpha_1} ( U^{d,\alpha_1} (\alpha, \beta) )
     -
     \pmb\eta_{\alpha_2} ( U^{d,\alpha_2} (\alpha, \beta) )
     =
      c B A
     =
         c  N( U^{d, \alpha_1} (\alpha, \beta) )
\]
where
\[
  c = \frac{2 r'}{\sqrt{ r'^2 + r^2 }}
    = 2  \sin 2(\alpha-\alpha_0)
    .
\]

  To prove 5., it is enough to show that
  the tangent to $U$ at $U(\alpha, \beta)$ is the plane spanned by orthonormal base
$ \{ N( \frac {\partial U}{\partial \alpha} (\alpha,\beta) ) ,
     N( \frac {\partial U}{\partial \beta} (\alpha,\beta) ) \}$
where
\begin{gather}
  \label{eq:tangent-almost-radial-lbasic}
     \left\|
     \NewNewBsmall
     \pm
     \NewNewEsmall
     N\left( \tfrac {\partial U}{\partial \alpha}
     (\alpha,\beta)
     \right)
     -
     N(U
     (\alpha,\beta)
     )
     \right\|
     \le
     2 \cos 2(\alpha-\alpha_0)
\\
  \label{eq:tangent-perpendicular-lbasic}
     N\left( \tfrac {\partial U}{\partial \beta}
     (\alpha,\beta)
     \right)
   =
   J(0,1) U(\alpha,\beta)
   .
\end{gather}

     \NewNewBsmall
     Here $\pm$ denotes the sign of $\alpha-\alpha_0$.
     \NewNewEsmall
   The two vectors are orthogonal by \eqref{eq:g12}.
     \NewNewBsmall
                   By \eqref{eq:tangent-rad} and
                   \eqref{eq:r-formula}--\eqref{eq:sqrt-formula}
                   \[
                     N\left( \tfrac {\partial U}{\partial \alpha}
                     \right)
                    =
                    \sin 2(\alpha-\alpha_0) BA +
                    \cos 2(\alpha-\alpha_0) BA'
                    .
                   \]
   Using $N(U) = B A$
   and
   $\| B A \| = 1 = \| B A' \|$
   we get
                   \[
                             \|
                             \pm
                             N( \tfrac {\partial U}{\partial \alpha} )
                             - N(U)
                             \|
                          \le  (1- \left| \sin 2(\alpha-\alpha_0) \right| ) + \cos  2(\alpha-\alpha_0)
                          \le 2 \cos  2(\alpha-\alpha_0)
                   \]
     \NewNewEsmall
     which is \eqref{eq:tangent-almost-radial-lbasic}.
Furthermore we have
\[
     N\left( \tfrac {\partial U}{\partial \beta}
     (\alpha,\beta)
     \right)
   =
     \frac {1}{\sqrt{ g_{22}}}
     \frac {\partial U
            (\alpha,\beta)
             }{\partial \beta}
   \overset{\eqref{eq:Uder-beta}}{ = }
   B' A
   \overset{\eqref{eq:B'}}{ = }
   J(0,1) B A
   =
   J(0,1) U(\alpha,\beta)
\]
which is \eqref{eq:tangent-perpendicular-lbasic}.

  6.
  The mass formula is directly obtained by integration.
  Since $g_{12}=0$, the $2$-volume element has a simple form.
\begin{multline*}
      \mass (V^{d,\alpha_0}_{t_1,t_2})
   =
      \Hausd^2
      S_{t_1,t_2}
   =
      \int_{[t_1, t_2]} \intd\alpha\
         \int_{[0,2\pi]} \intd\beta\
            \sqrt{ g_{11} \, g_{22} }
\\
   =
      2 \pi
      \int_{[t_1, t_2]} \intd\alpha\
          r \sqrt{ r'^2 + r^2 }
   \overset{\eqref{eq:sqrt-formula}}{
   =
   }
      2 \pi
      \int_{[t_1, t_2]} \intd\alpha\
           d  \cos ^{-2} 2(\alpha-\alpha_0)
\\
   =
      \pi d (\tan 2(t_2-\alpha_0) - \tan 2(t_1-\alpha_0) )
\end{multline*}
      If $\alpha_0 \notin [ t_1, t_2 ]$ then
      $\sgn \tan2(t_2-\alpha_0) = \sgn\tan2(t_1-\alpha_0)$
      and
      \[
         d  \, | \tan 2(t_i-\alpha_0) |
        =   \sqrt { \tfrac {d^2}{\cos^2 2(t_i-\alpha_0)} - d^2 }
        =     \sqrt {        r(t_i)^4       - d^2 }
      \]
      since
      $
          r(t_i)^2 = d / \cos 2(t_i -\alpha_0)
      $.
      This gives the mass in the form
      \[
          \pi \left| \sqrt { r(t_2)^4 - d^2 } - \sqrt { r(t_1)^4 - d^2 } \right|
          .
      \]
      The expression that contains $\rho$ is obtained by the Mean value theorem applied
      to function $q \mapsto \sqrt{ q^2 - d^2 }$
      on interval
      $[r(t_1)^2, r(t_2)^2]$
      or
      $[r(t_2)^2, r(t_1)^2]$.
      (Thus $\rho$ depends on $d$, $r(t_1)$ and $r(t_2)$
      but, naturally, not on $\alpha_0$.)
      Since obviously
      $S _ {t_1, t_2} \cap A _ {r(s_1)} ^ {r(s_2)} = S_{s_1,s_2} $
      we have
      \[
           V^{d,\alpha_0}_{t_1,t_2} ( G_2 ( A _ {r(s_1)} ^ {r(s_2)} ) )
        =
      \mass (V^{d,\alpha_0}_{s_1,s_2})
      .
      \]
\qed
\end{proof}
\end{theopargself}

\subsection{Mini-layer. Details about branching.}
\NewMovedHereSmallB
For $\rho>0$ and $\alpha \in \R$, denote
\NewChangeSmall
 \begin{equation}\label{eq:hhhhh}
     K(\rho, \alpha)
       =
          \{
             \rho (
                     \cos \alpha \cos \beta,
                     \sin \alpha \cos \beta,
                     \cos \alpha \sin \beta,
                     \sin \alpha \sin \beta
                  )
           \setcolon
              \beta \in \R
           \}
 \end{equation}
\NewMovedHereSmallE
\NewBrefsugSmall
    which is the circle of radius $\varrho$ parameterized by $\beta$ and oriented in $\R^4$
    by the choice of $\alpha$.
\NewErefsugSmall

   From the ring varifolds we construct
   two types of (mini-layer) varifolds:
   $V_1$ branching
\NewChangeSmall
   outwards
   and
   $V_2$ branching
\NewChangeSmall
   inwards
   (Figure~\ref{fig:mini-layers}).
   That is, $\vari V_1$ is supported on a number of circles
   \NewBrefsugSmall
   of type $K(\varrho,\alpha)$
   \NewErefsugSmall
   of
\NewChangeSmall
   smaller
   radius
   and twice as much circles
   \NewBrefsugSmall
   $K(\varrho,\alpha)$
   \NewErefsugSmall
   of
\NewChangeSmall
   larger
   radius.
   We carefully compute the densities of $\vari V_i$ on the circles
   and record the mass distribution.
   \NewBrefsugSmall
   The densities of $\vari V_i$ (see \eqref{eq:V1-layer-vari},
   \eqref{eq:V2-layer-vari})
   determine four constants denoted by $C,c$ with decorations.
   $C$ is the density on larger circles $K(r_2, \cdot)$, $c$ on
   smaller circles $K(r_1, \cdot)$.
   Tilde marks the ones related to $\vari V_1$ as opposed to $\vari V_2$.
   (The relation $C_{k,\gamma}= \tilde c_{k,\gamma}$ is best regarded as just a coincidence
   although it appears naturally from the manipulations with the objects and numbers.)
   \NewErefsugSmall

\def\DOpicture#1#2#3#4#5#6#7{\relax
  \def\RR{#1}\relax
  \def\step{#2}\relax
  \def\Rmult{#3}\relax
  \def\start{#4}\relax
  \def\stopA{#5}\relax
  \def\stopB{#6}\relax
  \def\forDATA{#7}\relax
  \def\colorSegments{black} %
  \def\colorCurved{black} %
  \pgfmathparse{0.82*exp(ln(\Rmult)*0.5)}\relax
  \xdef\RmultQQQ{\pgfmathresult}\relax
  \message{--tikzpicture--}\relax
  \begin{tikzpicture}[x=\SCALE,y=\SCALE, line cap=round]
  \relax
  \clip
  (-2,-2) rectangle (36.9,36);
  \relax
  \def\thickn{4pt}
  \forDATA
  {
  \pgfmathparse{\start+\step}\xdef\startstep{\pgfmathresult}
  \foreach \angle in {\start, \startstep,..., \stopA}
  {
  \draw [arrows=space-, line width={\thickness/2}, color=\colorSegments]
    (\angle:\RR) -- (\angle:{\RR*\Rmult});
  }
  \foreach \angle in {\start, \startstep,..., \stopB}
  {
  \draw [line width={\thickness/3}, color=\colorCurved]
    (\angle:{\RR/1}) .. controls ({\angle+\step/4}:{\RR*\RmultQQQ}) .. ({\angle+\step/2}:{\RR*\Rmult})
    ({\angle+\step}:{\RR/1}) .. controls ({\angle+3*\step/4}:{\RR*\RmultQQQ}) .. ({\angle+\step/2}:{\RR*\Rmult})
    ;
  }
  \message{^^J----- \RR : \step : \Rmult : \RmultQQQ ------- }
  \xdef\RR{\RR*\Rmult}
  \pgfmathparse{\step/2}
  \xdef\step{\pgfmathresult}
  }
  \end{tikzpicture}
}

  \begin{figure}[hbt]
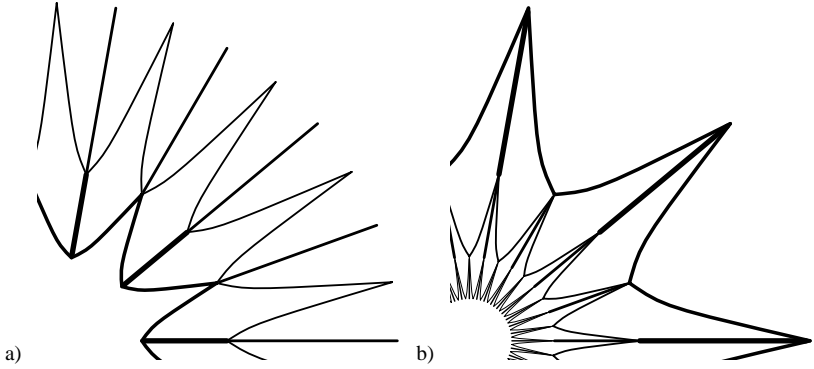

  \begin{center}
  \def\SCALE{56mm/45/1}
  \beginpgfgraphicnamed{varif2012-fig-minilayers-v2}
  \relax
   a) \DOpicture{9}{40}{2}{-40}{130}{130}{\foreach \thickness in {{\thickn}, {\thickn/1.9}}}
   b) \DOpicture{36}{40}{0.5}{-40}{130}{130}{\foreach \thickness in {\thickn, {\thickn/1.9}, {\thickn/1.9/1.9}}}\qquad
   \endpgfgraphicnamed
  \caption{\relax
          a) Two mini-layers branching outwards.
          b) Three mini-layers branching inwards.\relax
          \label{fig:mini-layers}}
  \end{center}
  \message{--figure--finito--.}
  \end{figure}

\begin{lemma}\label{l:basic_layer}
  \NewMovedSmallL
    Let $k\in \N$, $k>20$ and
    $\gamma \in (\pi/8,\pi/4)$
    be fixed.
    Let
\begin{align*}
    \sigma  &=
    \sigma_{k,\gamma} =
            \sqrt{
                  \frac{
                      \cos 2 \gamma
                  }{
                      \cos 2( \gamma - \pi / k )
                  }
             }
      \in (0,1)
      ,
\\
     \varepsilon &=  2 \cos 2( \gamma - \pi/k)
\end{align*}
   and
\begin{align*}
        \tilde
        C_ {k, \gamma}
              &=
                        4 \sin ( 2\gamma )
\\
        \tilde
        c_ {k, \gamma}
              &=
                2 \sin ( 2(\gamma - \pi /k ) )
                        +
                        2 \sin ( 2\gamma )
               =
                        4
                        \sin ( 2\gamma - \pi /k )
                        \cos (  \pi /k )
\\
        C_ {k, \gamma}
              &=
                        \tilde
                        c_ {k, \gamma}
\\
        c_ {k, \gamma}
              &=
                4 \sin ( 2(\gamma - \pi /k ) )
        .
\end{align*}

   Then,
   for every $r_2 > 0$ and for $r_1 = \sigma r_2$,
   there are rectifiable $2$-varifolds
   $V_1 = V_1^{r_1, r_2, k, \gamma}$,
   $V_2 = V_2^{r_1, r_2, k, \gamma}$
   in $\R^4$
   \NewBrefsugSmall
   \textup(see \eqref{eq:V1-layer} and \eqref{eq:V2-layer} for the definition\textup)
   \NewErefsugSmall
   such that
   $\spt \mu_{V_i} \subset A_{r_1}^{r_2}$,
   \begin{equation}\label{eq:Vi-spt}
      \spt
      V_i
      \subset
              G_2(A_{r_1}^{r_2})
              \cap
              G_{\RadAndOneThree}^{\,\varepsilon}
       ,
   \end{equation}
   \begin{align}\label{eq:mass-l-basic_layer}
                       4 \sin 2(\gamma-\pi/k) \cdot \pi ( (s_2)^2 - (s_1)^2 )
                   &\le
   \mass(V_i \mrest G_2( A  _ {s_1} ^ {s_2}) )
   =
   V_i ( G_2 ( A _ {s_1} ^ {s_2}))
\\
\notag
   &\le
       \frac{4}{
                      \sin 2( \gamma - \pi / k )
            }
      \pi ( (s_2)^2 - (s_1)^2 )
   \end{align}
   whenever $r_1 \le s_1 < s_2 \le r_2$,
   and
\begin{align}
  \label{eq:V1-layer-vari}
    \vari V_1 (X)
           &=
               \tilde C_ {k, \gamma}
                \,
                \Mucircles_ {r_2, 2k} (X)
             -
               \tilde c_ {k, \gamma}
                \,
                \Mucircles_ {r_1, k} (X)
 \\
 \label{eq:V2-layer-vari}
    \vari V_2 (X)
           &=
                C_ {k, \gamma}
                \,
                \Mucircles_ {r_2, k} (X)
             -
                c_ {k, \gamma}
                \,
                \Mucircles_ {r_1, 2k} (X)
\end{align}
   where
\textup(denoting $N(x) = x / \| x \|$
\NewBsmall
and $K(\varrho,\alpha)$ as in \eqref{eq:hhhhh}\relax
\NewEsmall
\textup)
\begin{equation}\label{eq:Bboundary}
        \Mucircles_ {\rho, k} (X)
              =
                \frac 1k
                \,
                \sum_{i=1}^k
                \int
                     _ { K( \rho, 2i \pi /k) }
                     X \cdot N \intd \Hausd^{2}
                     .
\end{equation}
\end{lemma}
\begin{proof}
   Let $d>0$ be such that
\begin{align}\label{eq:r1-layer}
    r_2  &=  \sqrt{ d \, / \cos 2 \gamma  }
\\
\label{eq:r2-layer}
    r_1  &=  \sigma r_2 = \sqrt{ d \, / \cos 2( \gamma - \pi / k ) }
    .
\end{align}

    Let
    $V ^ {d,\alpha_0} _ {t_1,t_2}$
    be as in Lemma~\ref{l:basic}, cf.~\eqref{eq:ring-def}
    ($\alpha_0 \in \R$ and $\alpha_0 - \pi/4 < t_1 < t_2 < \alpha_0 + \pi/4$).

    Let
\begin{align}
    \label{eq:V01-layer}
    V_{01}
     &=
             \sum_{i=1}^k
             \left (
               V ^ {d, \ ( 2i + 1 ) \pi /k - \gamma } _ { 2i  \pi /k  , \ ( 2i + 1 ) \pi /k }
               +
               V ^ {d, \ ( 2i + 1 ) \pi /k + \gamma } _ { ( 2i + 1 ) \pi /k , \ ( 2i + 2 ) \pi /k  }
             \right )
        ,
  \\
    \label{eq:V02-layer}
    V_{02}
     &=
             \sum_{i=1}^k
             \left (
               V ^ {d, \ 2i \pi /k - \gamma } _ { ( 2i - 1 ) \pi /k  , \ 2i \pi /k }
               +
               V ^ {d, \ 2i \pi /k + \gamma } _ {  2i \pi /k , \ ( 2i + 1 ) \pi /k  }
             \right )
\end{align}
  (see Figure~\ref{fig:mini-layers-withgrey}).
  \NewBrefsugSmall
  \NewNewBsmall
     The parameters of all $V^{d,\alpha_0}_{t_1,t_2}$ in \eqref{eq:V01-layer}, \eqref{eq:V02-layer}
     are so chosen that $r(t_i)=r^{d,\alpha_0}(t_i)$ from \eqref{eq:ralpha} attain
     exactly the values $r_1$, $r_2$, cf.\ \eqref{eq:r1-layer}, \eqref{eq:r2-layer}.
     Therefore all $V^{d,\alpha_0}_{t_1,t_2}$ are supported by $A_{r_1}^{r_2}$.
  \NewNewEsmall
  The difference between $V_{01}$ and $V_{02}$ is just a rotation which allows
  (together with $V_{00}$ below)
  proper alignment with
  the  neighbouring mini-layers as in Figure~\ref{fig:mini-layers-withgrey}.
  \NewErefsugSmall
  \begin{figure}[bt]
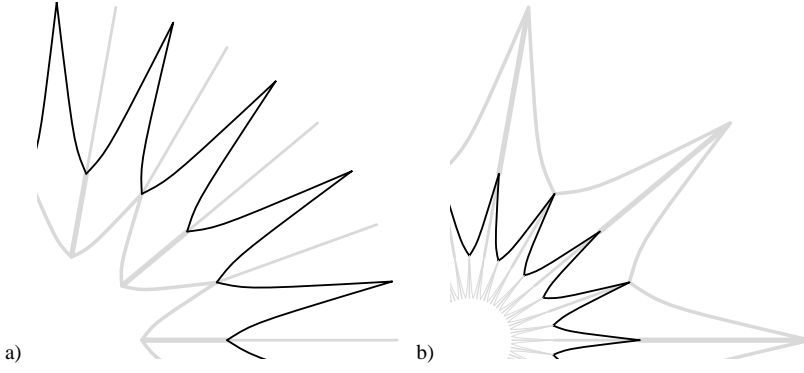

  \begin{center}
  \def\SCALE{56mm/45/1}
  \beginpgfgraphicnamed{varif2012-fig-minilayers-withgrey-v2}
  \relax
   \def\ggrey{black!15!white}
   a) \DOpicture{9}{40}{2}{-40}{130}{130}{\def\colorSegments{\ggrey}\foreach \thickness/\colorCurved in {{\thickn}/\ggrey, {\thickn/1.9}/black}}
   b) \DOpicture{36}{40}{0.5}{-40}{130}{130}{\def\colorSegments{\ggrey}\foreach \thickness/\colorCurved in {\thickn/\ggrey, {\thickn/1.9}/black, {\thickn/1.9/1.9}/\ggrey}}\qquad
   \endpgfgraphicnamed
  \caption{\relax
         The varifolds
          a) $V_{01}$ and
          b) $V_{02}$
          \eqref{eq:V01-layer}, \eqref{eq:V02-layer}
          on the gray background of Figure~\ref{fig:mini-layers}.
          (Note that a) and b) are not drawn and will not be used at the same scale.)
          \NewBrefsugSmall
          We are patching rings by actually patching pieces of the planar curve $r(\alpha)$ from \eqref{eq:ralpha}.
          The result is then rotated in $\R^4$ as indicated in by $\beta$ in \eqref{eq:UfromF}.
          The radial segments (they create planar annuli by the rotation)
          will be added later with (a proper density), see $V_{00}$ in \eqref{eq:V00}
          and \eqref{eq:V1-layer}, \eqref{eq:V2-layer}.\relax
          \NewErefsugSmall
          \label{fig:mini-layers-withgrey}}
  \end{center}
  \message{--figure--finito--.}
  \end{figure}
    \relax
    From \eqref{eq:ring-vari}, \eqref{eq:UU} and \eqref{eq:two-rings},
    we have
\begin{align}
   \label{eq:V01-layer-vari}
    \vari V_{01} (X)
           =
            {}
             &
                2 \sin ( 2\gamma )
                \,
                \sum_{i=1}^k
                \int
                     _ { K( r_2, ( 2i + 1 ) \pi /k) }
                     X \cdot N \intd \Hausd^{2}
                     \\
                     \notag
             &
             -
                2 \sin ( 2(\gamma - \pi /k ) )
                \,
                \sum_{i=1}^k
                \int
                     _ { K( r_1, ( 2i + 2 ) \pi /k) }
                     X \cdot N \intd \Hausd^{2}
             ,
\displaybreak[0]
\\
   \label{eq:V02-layer-vari}
    \vari V_{02} (X)
           =
            {}
             &
                2 \sin ( 2\gamma )
                \,
                \sum_{i=1}^k
                \int
                     _ { K( r_2, 2i \pi /k) }
                     X \cdot N \intd \Hausd^{2}
                     \\
                     \notag
             &
             -
                2 \sin ( 2(\gamma - \pi /k ) )
                \,
                \sum_{i=1}^k
                \int
                     _ { K( r_1, ( 2i + 1 ) \pi /k) }
                     X \cdot N \intd \Hausd^{2}
             .
\end{align}

    Let
\begin{equation}\label{eq:V00}
           V_{00}
           =
           V_{00}^{r_1, r_2, k}
           =
           \frac 1k
           \,
           \sum_{i=1}^k
              V_ {
                 1
                 ,
                 \,
                 r_1, r_2,
                 \,
                 (\cos 2i \pi /k, \, \sin 2i \pi /k)
                 }
\end{equation}
    where
     $
   V_ {1,r_1, r_2,(a,b)}
   =
   V_ {\Hausd^2 \mrest ( \Span \{ ae_1 + be_2 ,\, ae_3 + be_4 \} \cap A_{r_1}^{r_2}(\R^4) ) }
     $
        (see
        also \eqref{eq:V1rsab}, \eqref{eq:V2rscd}).
    Since
    $\Span \{ ae_1 + be_2 ,\, ae_3 + be_4 \}$ is a linear space invariant under multiplication
    by $\JOneThree$ (see~\eqref{eq:J01}),
    we have
\begin{equation}\label{eq:V00-spt}
  \spt V_{00}
  \subset
              G_2( A_{r_1}^{r_2} )
              \cap
              G_{\RadAndOneThree}^{\,0}
    \   .
\end{equation}
    Furthermore (cf.~\eqref{eq:var-smooth2} or Section~\ref{s:non-rect}),
\begin{align}
\notag
    \vari V_{00} (X)
           & =
           \frac 1k
           \,
           \sum_{i=1}^k
           \left(
                \int
                     _ { K( r_2,  2i \pi /k) }
                     X \cdot N \intd \Hausd^{2}
             -
                \int
                     _ { K( r_1,  2i \pi /k) }
                     X \cdot N \intd \Hausd^{2}
           \right)
\\
\label{eq:V00-vari-BB}
           & =
                 \Mucircles_{r_2, k} (X) - \Mucircles_{r_1,k} (X)
             .
\end{align}

    Let
\begin{align}
     \label{eq:V1-layer}
       V_1
          &=
           \frac 1k
           \,
           V_{01}
           +
           2 \sin ( 2\gamma )
           \,
           V_{00}
 \\
     \label{eq:V2-layer}
       V_2
          &=
           \frac 1k
           \,
           V_{02}
           +
           2 \sin ( 2(\gamma - \pi /k ) )
           \,
           V_{00}
           .
\end{align}
   Then the first variation of $V_1$ and $V_2$ is exactly as stated
   in \eqref{eq:V1-layer-vari}, \eqref{eq:V2-layer-vari}.
   Note that
   \[\
      \spt_{V_{01}}
      \cup
      \spt_{V_{02}}
      \subset
              G_{\RadAndOneThree}^{\,\varepsilon}
   \]
   by Lemma~\ref{l:basic}, 5.,
   and the same is true for planar varifold $V_{00}$,
   so also for
   $V_1$ and $V_2$.

   Let $ r_1 \le s_1 < s_2 \le r_2$.
   We claim that
   \begin{equation}\label{eq:masses}
    \mass ( V_{01} \mrest G_2( A _ {s_1} ^ {s_2}) )
    =
    \mass ( V_{02} \mrest G_2( A _ {s_1} ^ {s_2}) )
    =
    c 2 k \pi ( (s_2)^2 - (s_1)^2 )
   \end{equation}
   where
   \begin{equation}\label{eq:c-estimate}
                               \frac{1}{\sin 2\gamma}
                            \le
       c
    \le
       \frac{1}{
                      \sin 2( \gamma - \pi / k )
            }
      .
   \end{equation}

   Indeed,
   if $\rho  = \rho (d, s_1 , s_2) \in [s_1, s_1] \subset [r_1, r_2]$
   is as in
   Lemma~\ref{l:basic}, 6.,
   then \eqref{eq:masses} holds true with
   \[
       c
    =
       \frac{1}{\sqrt{1-\frac{d^2}{\rho^4}}}
    \le
       \frac{1}{\sqrt{1-\frac{d^2}{(r_1)^4}}}
    =
       \frac{1}{\sqrt{1-
                      \cos ^2 2( \gamma - \pi / k )
            }}
    =
       \frac{1}{
                      \sin  2( \gamma - \pi / k )
            }
       .
   \]
                                   On the other hand,
                                   \[
                                       c
                                    \ge
                                       \frac{1}{\sqrt{1-\frac{d^2}{(r_2)^4}}}
                                    =
                                       \frac{1}{\sqrt{1-\cos^2 2\gamma}}
                                    =
                                       \frac{1}{\sin 2\gamma}
                                       .
                                   \]

   We have exactly
   $
      \mass (
       V_{00} \mrest G_2 (A _ {s_1} ^ {s_2} )
       )
       =
       \pi ( (s_2)^2 - (s_1)^2 )
   $.
   Combining that with \eqref{eq:c-estimate}, we get
   (for $i = 1,2$)
   \begin{align*}
                                       4 \sin(2(\gamma-\pi/k)) \cdot \pi ( (s_2)^2 - (s_1)^2 )
                                   &\le
                                     \left(
                                       \tfrac{2}{\sin 2\gamma}
                                       + 2 \sin(2(\gamma- \pi/k))
                                     \right)
                                     \pi ( (s_2)^2 - (s_1)^2 )
                              \\
                                   &\le
   \mass(V_i  \mrest G_2 (A _ {s_1} ^ {s_2} ) )
   \\
   &\le
     \left(
       \tfrac{2}{
                      \sin 2( \gamma - \pi / k )
            }
      + 2 \sin(2\gamma)
     \right)
     \pi ( (s_2)^2 - (s_1)^2 )
 \\
   &\le
       \tfrac{4}{
                      \sin 2( \gamma - \pi / k )
            }
      \pi ( (s_2)^2 - (s_1)^2 )
      .
   \end{align*}
   which is \eqref{eq:mass-l-basic_layer}.
\qed
\end{proof}

\subsection{Layers.}
   Recall now that $F$ is defined by~\eqref{eq:F-def}
   (see also \eqref{eq:F-radial} and \eqref{eq:F-exchange}).
\begin{lemma}\label{l:system-A}
     If $0< R_1 < R_2 < R_3 < R_4 <\infty$
     and $\varepsilon > 0$
     then there is
     $ c \in ( 1-\varepsilon, 1) $
     and
     a rectifiable 2-varifold $V$ with
     $\spt \mu_V \subset A_{ R_1 }^{ R_4 }$,
\begin{align}
\label{eq:sptV}
       \spt V
    &\subset
              G_2(
                A_{ R_1 }^{ R_2 }
                \cup
                A_{ R_3 }^{ R_4 }
               )
              \cap
              G_{\RadAndOneThree}^{\,\varepsilon}
\\
\notag
    &\phantom{{}\subset{}}
         {\cup}\  %
              G_2(
                A_{ R_2 }^{ R_3 }
               )
              \cap
              G_{\RadAndOneThree}^{\,0}
\\
\label{eq:sptV-simple}
    &\subset
              G_2(A_{ R_1 }^{ R_4 })
              \cap
              G_{\RadAndOneThree}^{\,\varepsilon}
   ,
\end{align}
\begin{equation}\label{eq:V-mass}
     ( 1 - \varepsilon ) \pi ( (s_2)^2  -  (s_1)^2 )
     <
     \mass(V \mrest G_2 ( A _ {s_1} ^ {s_2} ) )
     <
     ( 1 + \varepsilon ) \pi ( (s_2)^2  -  (s_1)^2 )
\end{equation}
   whenever $R_1 \le s_1 < s_2 \le R_4$,
     and
\begin{equation}\label{eq:layer-V-vari}
   \vari V (X)
 =
     \Mucircles_{R_4, \infty} (X)
 -
   c \Mucircles_{R_1, \infty} (X)
\end{equation}
      where \textup(with $N(x)=x \, / \left\|x\right\|$\textup)
\begin{equation}\label{eq:Binfty}
      \Mucircles_{\rho, \infty} (X)
   =
        \int\limits
         _{F(( \rho  \cdot S_1(\R^2))\times S_1(\R^2))}  X\cdot N
                     \ \ \intd\ \frac{\Hausd^2}{ 2 \pi \rho }
   .
\end{equation}
\end{lemma}
\NewB
\begin{remark}
        Note that $\Mucircles_{\rho, \infty}$ is a vector measure on the scaled Clifford torus
        $F(( \rho  \cdot S_1(\R^2))\times S_1(\R^2))$ that is
        \NewBrefsugSmall
        {\em uniformly} diffuse
        \NewErefsugSmall
        in the sense that the associated total variation measure
        is just a constant multiple of the Hausdorff measure.
        This comes from the properties of the Clifford torus and from how uniformly
        $\Mucircles_{\rho,k}$ from \eqref{eq:Bboundary} is distributed on the ``parallel'' circles.
        (For their weak convergence see also \eqref{eq:Bconverge} below.)

        This will be important later for compatibility
        on the interface,
        when $V$ from Lemma \ref{l:system-A}
        is used together with similar but different varifold $V$ from Lemma~\ref{l:system-B}.
  \qqeedd
\end{remark}

        Before giving the formal proof of Lemma~\ref{l:system-A}, we explain how the varifold $V$ (the layer) is constructed.

        The space between $R_1$ and $r^{(n_0)} \in (R_1, R_2]$
        is occupied by an infinite sequence of varifolds (mini-layers)
        from Lemma~\ref{l:basic_layer}
        and Figure~\ref{fig:mini-layers}b)
        that are branching towards the inner interface which is
        the Clifford torus at radius $R_1$.
        Here $n_0$ is a technical index
        (to be explained later)
        and $r^{(n_0)}$ is chosen at our convenience for using Lemma~\ref{l:basic_layer}.
        The mini-layers are indexed by $n\ge n_0$ and each of them lives between suitably
        defined radii $r^{(n+1)}$ and $r^{(n)}$
        where $r^{(n+1)} < r^{(n)}$.
        The connections at radii $r^{(n)}$ ($n>n_0$)
        (i.e., the branching) can also be seen in Figure~\ref{fig:mini-layers}b).

        Likewise, the space between $R^{(n_0)} \in [R_3, R_4)$ and $R_4$
        is occupied by an infinite sequence of mini-layers from
        Figure~\ref{fig:mini-layers}a)
        that are branching towards the {\em outer} interface which is
        the Clifford torus at radius $R_4$.
        Each of them lives between suitably
        defined radii $R^{(n)}$ and $R^{(n+1)}$
        where $R^{(n)}< R^{(n+1)}$.

        The space between $r^{(n_0)}$ and $R^{(n_0)}$ is bridged by a varifold supported on a finite number of
        planar annuli,
        which is the term $c_1 V_{00}$ in the definition of $V$ below.
        (I~have
        received a question about the purpose of $c_1 V_{00}$. Obviously, the space between $r^{(n_0)}$ and $R^{(n_0)}$
        should not be left empty if we wish to have a stationary varifold.
        In principle it would be possible to assume $r^{(n_0)} = R_2 = R_3 = R^{(n_0)}$
        and avoid the term $c_1 V_{00}$ but we have chosen an easier way.
        Moreover, for the proof of Theorem~\ref{thm:conical} it is useful to have $R_3/R_2$ very large
        and $c_1 V_{00}$ is not only the most easy but also the most natural and most intuitive candidate to fill the space.
        Though, mini-layers have enough flexibility to replace its role.)

        As it is indicated above, we use infinite sequences of mini-layers and we have to emphasize that
        the corresponding sequences of parameters for Lemma~\ref{l:basic_layer} need (and fortunately can)
        be chosen so that both
        1) products of density coefficient ratios are positive
        and
        2) the product of radii ratios is positive.

        This then allows to choose the technical index
        $n_0$
        large enough
        to obtain 1) estimate \eqref{eq:V-mass}
        2) an infinite number of mini-layers that fits between $R_1$ and $R_2$
        ($R_3$ and $R_4$, respectivelly).

        The meaning of index $n_0$ is the following: out of a sequence of candidate mini-layers
        (which are indexed by $n$) we forget the first $n_0$ of them and use only the tail of the sequence.
        After $n_0$ is known
        we decide the radii to which we scale the mini-layers as well the densities that we apply to them.

\begin{remark}
        Without giving details we note that $n_0$ has to be chosen large if $\varepsilon$ is small.
        This results in the density coefficients of $V$ being bounded from above by about $1/2^{n_0}$
        (and $n_0\to \infty$ when layers closer to $0\in \R^4$ are considered in our application of the lemma)
        which is not much desired
        and implies that our {\em tangent} varifolds will be non-rectifiable.
        Actually,
        the presence of a large number of points of {\em small} density
        is unavoidable if we want to obtain
        {\em non-conical}
        tangents,
        see
        Lemma~\ref{l:abouttangents} and its references.

        Furthermore, the two-dimensional density at the points of the interface
        (i.e., at radii $R_1$ and $R_4$)
        of each layer
        will be zero. Nevertheless, this is negligible in measure (as measured by the varifold)
        and does not prevent us from
        constructing a rectifiable varifold. We leave open whether points of {\em zero} two-dimensional density
        can be completely avoided in a varifold with non-conical tangents.
  \qqeedd
\end{remark}

\NewE

\begin{proof}[of Lemma~\ref{l:system-A}]
        Choose
        $k \indn = 100 \cdot 2^n$
        and
        $\gamma \indn = \pi/4 - \pi/ \sqrt{ k \indn }$.

        With
        $ C_ {k, \gamma}$, $ c_ {k, \gamma}$
        $ \tilde C_ {k, \gamma}$, $ \tilde c_ {k, \gamma}$
        and
        $ \sigma_ {k, \gamma}$ as in Lemma~\ref{l:basic_layer}
        we have
\begin{align*}
    1
 &\ge
    \frac{
        c_ {k \indn, \gamma \indn}
         }{
        C_ {k \indn, \gamma \indn}
         }
             \ge
                 \sin ( 2(\gamma \indn - \pi /k \indn ) )
             \ge 1 - 8 \pi^2 / k \indn
             > 0
     ,
\\
    1
 &\ge
    \frac{
       \tilde
        c_ {k \indn, \gamma \indn}
         }{
       \tilde
        C_ {k \indn, \gamma \indn}
         }
             \ge
                 \sin ( 2\gamma \indn - \pi /k \indn )
                 \cos(\pi/k \indn)
             \ge
                 1 - 5 \pi^2 / k \indn
             > 0
             .
\end{align*}
        Hence
\[
    \prod_{n=1}^{\infty}
    \frac{
        c_ {k \indn, \gamma \indn}
         }{
        C_ {k \indn, \gamma \indn}
         }
    \in (0,1)
     ,
\qquad
\qquad
    \prod_{n=1}^{\infty}
    \frac{
      \tilde
        c_ {k \indn, \gamma \indn}
         }{
      \tilde
        C_ {k \indn, \gamma \indn}
         }
    \in (0,1)
  .
\]
   Furthermore
\begin{multline*}
    0
\le
    1
    -
    \sigma_{k \indn, \gamma \indn}
      ^2
=
    1
    -
    \frac{
                 \sin  2\pi / \sqrt{ k \indn }
         }{
                 \sin  ( 2\pi / k \indn + 2\pi / \sqrt{ k \indn } )
         }
  \\
=
    \frac{
            2 \cos (  \pi / k \indn + 2\pi / \sqrt{ k \indn } )  \sin \pi/k\indn
         }{
                 \sin  ( 2\pi / k \indn + 2\pi / \sqrt{ k \indn } )
         }
  \le
    \pi
    \frac{
              \pi / k \indn
         }{
                  2\pi / k \indn + 2\pi / \sqrt{ k \indn }
         }
  \le
          \frac \pi 2
          \frac 1 { \sqrt { k \indn } }
   ,
\end{multline*}
hence
\[
    \left(
      \prod_{n=1}^{\infty}
            \sigma_{k \indn, \gamma \indn}
    \right)
            ^2
   =
      \prod_{n=1}^{\infty}
            \sigma_{k \indn, \gamma \indn} ^2
    \in (0,1)
    .
\]
    Choose $n_0\in \N$ so that
    (for $n\ge n_0$)
    \[
        \varepsilon_n := 2 \cos 2(\gamma \indn - \pi /k \indn) < \varepsilon
        ,
    \]
\begin{equation}\label{eq:sinus}
    \sin 2(\gamma \indn - \pi / k \indn ) > 1 - \varepsilon /3
    ,
\end{equation}
\begin{equation}\label{eq:M}
   M
   :=
               \tfrac{1}{
                    4 \sin ( 2\gamma \indn [_0] - \pi / k \indn [_0]) \cos ( \pi / k \indn [_0] )
                }
               \tfrac{4}{
                              \sin 2( \gamma \indn [ _0 ] - \pi / k \indn [ _0 ])
                    }
    < 1 + \varepsilon
    ,
\end{equation}
\begin{align}
  \notag
    c_1:=
    \prod_{n=n_0}^{\infty}
    \frac{
      \tilde
        c_ {k \indn, \gamma \indn}
         }{
      \tilde
        C_ {k \indn, \gamma \indn}
         }
    \in (1 - \varepsilon/3,1)
    ,
\\
  \notag
    c_2:=
    \prod_{n=n_0}^{\infty}
    \frac{
        c_ {k \indn, \gamma \indn}
         }{
        C_ {k \indn, \gamma \indn}
         }
    \in (1 - \varepsilon/3,1)
\\
    \noalign{\noindent and}
  \notag
    \sigma :=
      \prod_{n=n_0}^{\infty}
            \sigma_{k \indn, \gamma \indn}
    \in ( \max( R_1/R_2, R_3/R_4) ,1)
    .
\end{align}
    Let $r \indn[ _0 ] :=  R_1  / \sigma$, $R \indn[ _0 ] := \sigma  R_4 $,
    and then
    inductively   %
    $r \indn[ +1 ] :=  \sigma_{k \indn, \gamma \indn} r \indn$,
    $R \indn[ +1 ] := R \indn /  \sigma_{k \indn, \gamma \indn}$.
    Then
    $\lim_{n\to \infty} r \indn = R_1$,
    $\lim_{n\to \infty} R \indn = R_4$,
\begin{gather*}
      R_1  < r \indn [ _0 ] \le R_2 \le R_3 \le  R \indn [ _0 ] <  R_4
      ,
 \\
      R_1  < \dots < r \indn [ _0 +2] < r \indn [ _0 +1] <r \indn [ _0 ]
           \le R \indn [ _0 ] < R \indn [ _0 +1] < R \indn [ _0 +2] < \dots <  R_4
      .
\end{gather*}
       Let
\begin{align}
  \notag
    c_{1,n}:&=
    \prod_{m=n}^{\infty}
    \frac{
      \tilde
        c_ {k \indX[m], \gamma \indX[m]}
         }{
      \tilde
        C_ {k \indX[m], \gamma \indX[m]}
         }
    \in (0,1)
   \qquad \qquad \text{(hence $c_1 = c_{1,n_0}$)}
\\
  \notag
    c_{2,n}:&=
    \prod_{m=n_0}^{n-1}
    \frac{
        c_ {k \indX[m], \gamma \indX[m]}
         }{
        C_ {k \indX[m], \gamma \indX[m]}
         }
    \in (0,1]
   \qquad \qquad \text{($c_{2,n_0}:=1$; $c_2 = c_{2,\infty}$)}
   ,
\end{align}
        and let
        $V_1^{r,s,k,\gamma}$, $V_2^{r,s,k,\gamma}$
        be as in
        Lemma~\ref{l:basic_layer}
        and
        $V_{00}=V_{00}^{ r \indn [ _0 ], R \indn [ _0 ], k \indn}$
        is as in~\eqref{eq:V00}.
        Let
\[  %
   V =
        \sum_{n=n_0}^\infty
             \frac{
                     c_1
                     \,
                     c_{2,n}
                  }{
                        C_ {k \indn, \gamma \indn }
                  }
             V_2^{ r \indn [ +1 ], r \indn, k \indn, \gamma \indn }
        +
        c_1
         V_{00}
        +
        \sum_{n=n_0}^\infty
             \frac{
                     c_{1,n+1}
                  }{
                     \tilde
                        C_ {k \indn, \gamma \indn }
                  }
              V_1^{ R \indn, R \indn [+1], k \indn, \gamma \indn }
\]  %
     and
\[  %
   V_m =
        \sum_{n=n_0}^m
             \frac{
                     c_1
                     \,
                     c_{2,n}
                  }{
                        C_ {k \indn, \gamma \indn }
                  }
             V_2^{ r \indn [ +1 ], r \indn, k \indn, \gamma \indn }
        +
        c_1
         V_{00}
        +
        \sum_{n=n_0}^m
             \frac{
                     c_{1,n+1}
                  }{
                     \tilde
                        C_ {k \indn, \gamma \indn }
                  }
              V_1^{ R \indn, R \indn [+1], k \indn, \gamma \indn }
     .
\]  %
        Then
        \eqref{eq:sptV} can be obtained from \eqref{eq:Vi-spt} and \eqref{eq:V00-spt}.

        Denote also
\[  %
   V - V_m :=
        \sum_{n=m+1}^\infty
             \frac{
                     c_1
                     \,
                     c_{2,n}
                  }{
                        C_ {k \indn, \gamma \indn }
                  }
             V_2^{ r \indn [ +1 ], r \indn, k \indn, \gamma \indn }
        +
        \sum_{n=m+1}^\infty
             \frac{
                     c_{1,n+1}
                  }{
                     \tilde
                        C_ {k \indn, \gamma \indn }
                  }
              V_1^{ R \indn, R \indn [+1], k \indn, \gamma \indn }
     .
\]  %

Note that
\begin{align*}
             \frac{
                     c_1
                     \,
                     c_{2,n}
                  }{
                        C_ {k \indn, \gamma \indn }
                  }
              \frac{4}{
                              \sin 2( \gamma \indn - \pi / k \indn)
                    }
           &\le M,
                \qquad\qquad n\ge n_0,
\\
              c_1
            \le
                1
           &\le
                M,
\\
             \frac{
                     c_{1,n+1}
                  }{
                     \tilde
                        C_ {k \indn, \gamma \indn }
                  }
              \frac{4}{
                              \sin 2( \gamma \indn - \pi / k \indn)
                    }
           &\le
               M,
                \qquad\qquad n\ge n_0
           .
\end{align*}
     Hence, by \eqref{eq:mass-l-basic_layer}
     and \eqref{eq:M},
\begin{align}\label{eq:Vmass}
 \mass ( V )
 \le
   {}
   &     \sum_{n=n_0}^\infty
             M \pi
                ( (r \indn )^2 - (r  \indn [ +1 ])^2 )
\\\notag  &
        +
             M \pi
                ( (R \indX [ n_0 ] )^2 - (r  \indX [ n_0 ] )^2 )
        +
        \sum_{n=n_0}^\infty
             M \pi
                ( (R \indn [ +1 ] )^2 - (R  \indn)^2 )
\\\notag
      =
       {}
         &
             M \pi
                    ( (R_4)^2 - (R_1)^2 )
      \le
             ( 1 + \varepsilon) \pi
                    ( (R_4)^2 - (R_1)^2 )
       .
\end{align}
In particular, $V$ is a Radon measure.
Therefore $V$ is a varifold, obviously rectifiable.
Moreover, $ \mass (V-V_n) \to 0$ as $n \to \infty$.

Note also that
\begin{align*}
             \frac{
                     4
                     c_1
                     \,
                     c_{2,n}
                  }{
                        C_ {k \indn, \gamma \indn }
                  }
           &\ge
                     c_1c_2,
                \qquad\qquad n\ge n_0,
\\
              c_1
           &\ge
                     c_1c_2,
\\
             \frac{
                     4
                     c_{1,n+1}
                  }{
                     \tilde
                        C_ {k \indn, \gamma \indn }
                  }
           &\ge
                     c_1c_2,
                \qquad\qquad n\ge n_0
           .
\end{align*}
        Again by \eqref{eq:mass-l-basic_layer} (and \eqref{eq:sinus}), we get
\begin{align}\label{eq:Vmass2}
 \mass ( V )
 \ge
   {}
   &     \sum_{n=n_0}^\infty
            (1-\varepsilon/3)
             c_1c_2  \pi
                ( (r \indn )^2 - (r  \indn [ +1 ])^2 )
\\\notag  &
        +
            (1-\varepsilon/3)
             c_1c_2  \pi
                ( (R \indX [ n_0 ] )^2 - (r  \indX [ n_0 ] )^2 )
\\\notag  &
        +
        \sum_{n=n_0}^\infty
            (1-\varepsilon/3)
             c_1c_2  \pi
                ( (R \indn [ +1 ] )^2 - (R  \indn)^2 )
\\\notag  =
          {}
          &
            (1-\varepsilon/3)
             c_1c_2  \pi
                    ( (R_4)^2 - (R_1)^2 )
      \ge
             ( 1 - \varepsilon) \pi
                    ( (R_4)^2 - (R_1)^2 )
       .
\end{align}
    From \eqref{eq:Vmass} and \eqref{eq:Vmass2}, \eqref{eq:V-mass} follows
    in the special case $s_1=R_1$, $s_2=R_4$.
    (Note that a special case $s_1=r_1$, $s_2=r_2$ of \eqref{eq:mass-l-basic_layer} was used.)
    Proof of the general case $R_1 \le s_1 < s_2 \le R_4$ of  \eqref{eq:V-mass}
    is similar, with the
    following differences: a) some of the terms in \eqref{eq:Vmass},
    \eqref{eq:Vmass2} might be replaced by $0$, and
    b) some (at most two) of the terms might be ``cut'' to a smaller
    span between radii; the general case of \eqref{eq:mass-l-basic_layer} is
    used in such a case.
    For example, \eqref{eq:Vmass2} is to be replaced by
\begin{align}\label{eq:Vmass2-general}
 \mass ( V \mrest G_2( A _ {s_1} ^ {s_2} ))
 \ge
   {}
   &     \sum_{n=n_0}^\infty
            (1-\varepsilon/3)
             c_1c_2  \pi
                ( (\widehat{ r \indn })^2 - (\widehat{ r  \indn [ +1 ]})^2 )
\\\notag  &
        +
            (1-\varepsilon/3)
             c_1c_2  \pi
                ( (\widehat{ R \indX [ n_0 ]} )^2 - (\widehat{ r  \indX [ n_0 ]} )^2 )
\\\notag  &
        +
        \sum_{n=n_0}^\infty
            (1-\varepsilon/3)
             c_1c_2  \pi
                ( (\widehat{ R \indn [ +1 ]} )^2 - (\widehat{ R  \indn})^2 )
\\\notag  =
          {}
          &
            (1-\varepsilon/3)
             c_1c_2  \pi
                    ( (s_2)^2 - (s_1)^2 )
      \ge
             ( 1 - \varepsilon) \pi
                    ( (s_2)^2 - (s_1)^2 )
       .
\end{align}
        where $\widehat{\rho} = \min (\max (s_1, \rho), s_2)$.

        We have
        $V_{n_0 - 1} = c_1 V_{00}$ and, by~\eqref{eq:V00-vari-BB},
        \[
             \vari V_{n_0 - 1}
           =
                 c_1 \Mucircles_{ R \indn [ _0 ] , k \indn [_0] }
                 -
                 c_1 \Mucircles_{ r \indn [ _0 ] , k \indn [_0] }
           =
                 c_{1,n_0} \Mucircles_{ R \indn [ _0 ] , k \indn [_0] }
                 -
                 c_1\,c_{2,n_0} \Mucircles_{ r \indn [ _0 ] , k \indn [_0] }
        \]
        where $\Mucircles_{\rho,k}$ is as in \eqref{eq:Bboundary}.
        Using \eqref{eq:V1-layer-vari} \eqref{eq:V2-layer-vari} we obtain by induction
\begin{equation}\label{eq:Vn-vari}
 \vari V_n
 =
   c_{1, n+1}
     \Mucircles_ {R \indn [+1], k\indn [+1] }
     -
   c_1 \,
   c_{2, n+1}
     \Mucircles_ {r \indn [+1], k\indn [+1] }
   .
\end{equation}
Indeed, for $n\ge n_0$,
\begin{multline*}
 \vari V_{n}
 =
             \frac{
                     c_{1,n+1}
                  }{
                     \tilde
                        C_ {k \indn, \gamma \indn }
                  }
              (
               \tilde C_ {k \indn, \gamma \indn}
                \,
                \Mucircles_ { R \indn [+1] , 2k \indn}
             -
               \tilde c_ {k \indn, \gamma \indn}
                \,
                \Mucircles_ { R \indn , k \indn}
              )
  +
\\
   c_{1, n}
     \Mucircles_ {R \indn , k\indn }
     -
   c_1 \,
   c_{2, n}
     \Mucircles_ {r \indn , k\indn }
  +
\\
             \frac{
                     c_1
                     \,
                     c_{2,n}
                  }{
                        C_ {k \indn, \gamma \indn }
                  }
              (
                C_ {k \indn, \gamma \indn}
                \,
                \Mucircles_ {r \indn, k \indn}
             -
                c_ {k \indn, \gamma \indn}
                \,
                \Mucircles_ {r \indn [ +1 ] , 2k \indn}
              )
\\
 =
   c_{1, n+1}
     \Mucircles_ {R \indn [+1], k\indn [+1] }
     -
   c_1 \,
   c_{2, n+1}
     \Mucircles_ {r \indn [+1], k\indn [+1] }
   .
\end{multline*}

It is easy to verify that, for every smooth vector field $X$,
\begin{align}
\label{eq:Bconverge}
     \Mucircles_ {R \indn [+1], k\indn [+1] }
     (X)
  \to
     \Mucircles_ {R _4, \infty}
     (X)
\\
\noalign{\noindent and}
\label{eq:Bconverge2}
     \Mucircles_ {r \indn [+1], k\indn [+1] }
     (X)
  \to
     \Mucircles_ {R _1, \infty}
     (X)
   .
\end{align}
        \NewBrefsugSmall
Indeed, the (local) uniform continuity of $X$ can be used in the same way as when proving the simple planar exercise
with Dirac masses
  $
  \frac 1 k \sum_{i=1}^k \delta_{\left( \frac i k, \frac 1 k \right)}  \overset{w}\to \Hausd ^1 \mrest \left([0,1]\times \{0\}\right)
  $.
        \NewErefsugSmall
On the other hand,
\[
   \left|
    \vari V (X) - \vari V_n (X)
   \right|
    \overset{\eqref{eq:var-def}}{
    =
    }
   \left|
    \int \diver_S X(x) \intd (V-V_n) (x,S)
   \right|
    \le
    \| X \| _{C^1} \cdot \mass (V - V_n)
    \to
    0
\]
    as $n\to \infty$.
    From \eqref{eq:Vn-vari}
            \NewBsmall
    and \eqref{eq:Bconverge}, \eqref{eq:Bconverge2}
            \NewEsmall
    we therefore obtain
    the formula
            \NewBsmall
    \eqref{eq:layer-V-vari}
            \NewEsmall
    for the first variation of $V$,
    with $c := \lim c_1\, c_{2,n} = c_1\, c_2 \in (1-\varepsilon,1)$.
\qed
\end{proof}
\begin{lemma}\label{l:system-B}
     If $0< R_1 < R_2 < R_3 < R_4 <\infty$
     and $\varepsilon > 0$
     then there is
     $ c \in ( 1-\varepsilon, 1) $
     and
     a rectifiable 2-varifold $V$ with
     $\spt \mu_V \subset A_{ R_1 }^{ R_4 }$,
\begin{align}
\label{eq:sptV-B}
       \spt V
    &\subset
              G_2(
                A_{ R_1 }^{ R_2 }
                \cup
                A_{ R_3 }^{ R_4 }
               )
              \cap
              G_{\RadAndOneTwo}^{\,\varepsilon}
\\
\notag
    &\phantom{{}\subset{}}
         {\cup}\  %
              G_2(
                A_{ R_2 }^{ R_3 }
               )
              \cap
              G_{\RadAndOneTwo}^{\,0}
\\
\label{eq:sptV-B-simple}
    &\subset
              G_2(A_{ R_1 }^{ R_4 })
              \cap
              G_{\RadAndOneTwo}^{\,\varepsilon}
   ,
\end{align}
\begin{equation}\label{eq:V-mass-B}
     ( 1 - \varepsilon ) \pi ( (s_2)^2  -  (s_1)^2 )
     <
     \mass(V \mrest G_2 ( A _ {s_1} ^ {s_2} ) )
     <
     ( 1 + \varepsilon ) \pi ( (s_2)^2  -  (s_1)^2 )
\end{equation}
   whenever $R_1 \le s_1 < s_2 \le R_4$,
     and
\begin{equation}\label{eq:layer-V-vari-B}
   \vari V (X)
 =
     \Mucircles_{R_4, \infty} (X)
 -
   c \Mucircles_{R_1, \infty} (X)
\end{equation}
      where $\Mucircles_{\rho, \infty}$ is as in \eqref{eq:Binfty}.
\end{lemma}
\begin{proof}
   The statement is the same as in Lemma~\ref{l:system-A}, with the exception
   of a change of coordinates in \eqref{eq:sptV-B} ---
   we show that
   it is enough to exchange
   coordinates $x_2$ and $x_3$.
   Let
   $\phi(x_1, x_2, x_3, x_4) = \phi (x_1, x_3, x_2, x_4)$, ($(x_1, x_2, x_3, x_4) \in \R^4$),
   and
   $\Phi( x, S)= (\phi (x), \phi(S))$ ($(x,S) \in G_2(\R^4)$).
   Then
   $
     \phi(\JOneThree x)
     =
     \JOneTwo \phi(x)
   $
   and $
   \Phi(G_{\RadAndOneThree}^{\,\varepsilon})
   =
        G_{\RadAndOneTwo}^{\,\varepsilon}
   $
   (cf.~\eqref{eq:Grad}).
   The domain of integration
   in \eqref{eq:Binfty}
   (which is parameterized by $F$)
   does not change under $\phi$:
   $\phi(F((\rho a, \rho b),(c,d)))
   \overset{\eqref{eq:F-exchange}}{=}
   F((c,d),(\rho a, \rho b))
   \overset{\eqref{eq:F-radial}}{=}
   F((\rho c, \rho d),(a,b))$.
           \NewNewBsmall
   Since $\phi$ is an isometry, it also preserves Hausdorff measure
   in \eqref{eq:Binfty}.
           \NewNewEsmall
   Therefore, if $\tilde V$ is as in  Lemma~\ref{l:system-A},
   then $V := \phi_\ImageVarifold \tilde V = \Phi_\ImageMeasureFV \tilde V$ is a varifold with required properties.
\qed
\end{proof}
\begin{lemma}
   If $V$ is as in Lemma~\ref{l:system-A} or Lemma~\ref{l:system-B} and $r > 0$
   then
   \begin{equation}\label{eq:nullsphere}
        \mass ( V \mrest G_2( S_r (\R^4))) = 0
      .
   \end{equation}
\end{lemma}
\begin{proof}
   For every $0 < \varepsilon_1 < r$ we heave by~\eqref{eq:V-mass}, \eqref{eq:V-mass-B},
\begin{equation*}
    \mass ( V \mrest G_2( A_{r-\varepsilon_1}^{r+\varepsilon_1} ))
    \le
    (1+\varepsilon) \pi ((r+\varepsilon_1)^2 - (r-\varepsilon_1)^2) \to 0
    .
\end{equation*}
\qed
\end{proof}
   We do the last step of our construction of a stationary rectifiable varifold in the next section.

\section{Two variants of the main result}\label{s:variants}
\begin{theorem}\label{thm:nonconical}
   There is a stationary rectifiable $2$-varifold $V$
   in $\R^4$ that has a non-conical (hence non-unique) tangent
   at\/ $0$
   and\/ $0<\theta^2(V, 0)<\infty$.
\end{theorem}
        \NewBrefsug
      The proof is built around the idea of alternating layers of two types of varifolds as in our
      non-rectifiable example in Section~\ref{s:non-rect}.
      For each layer, the varifold of Section~\ref{s:non-rect} is replaced by its rectifiable ``approximation''
      from Lemma~\ref{l:system-A} and Lemma~\ref{l:system-B}.
      However this introduces some excess and therefore the density coefficients must be calculated accordingly
      and we have to take care to get positive density at the origin, which means we have to estimate yet another infinite product.

      As we emphasised above, it is important that the first variations of the layer varifolds is a measure
      (vector measure with the radial directions) {\em uniformly}
      distributed on the interfaces and therefore it is compatible for our two types of layers which differ by a rotation.
      This important feature is shared with Section~\ref{s:non-rect}.
        \NewErefsug
\begin{proof}
{\bfseries 1. The varifold $V$.}
     For $0< R_1 < R_2 < R_3 < R_4 <\infty$
     and $\varepsilon > 0$
     let
\[
     V^1_{  R_1, R_2, R_3, R_4, \varepsilon}
     \qquad
     \text{
     and
     }
     \qquad
      c^1_{  R_1, R_2, R_3, R_4, \varepsilon}
      \in (1 - \varepsilon, 1)
\]
     denote the varifold and the number from Lemma~\ref{l:system-A}.
     Let
\[
     V^2_{  R_1, R_2, R_3, R_4, \varepsilon}
     \qquad
     \text{
     and
     }
     \qquad
      c^2_{  R_1, R_2, R_3, R_4, \varepsilon}
      \in (1 - \varepsilon, 1)
\]
     denote the varifold and the number from Lemma~\ref{l:system-B}.

      For $n\in \Z$, let
\begin{align*}
      \varepsilon \indn &= 1/4(n^2+1)
      \\
      R_1 \indn &= 2^{-n}
      \\
      R_2 \indn &= (1+\varepsilon \indn ) 2^{-n}
      \\
      R_3 \indn &= (1-\varepsilon \indn ) 2^{-n+1}
      \\
      R_4 \indn &= 2^{-n+1}
            =  R_1 \indn [-1]
      .
\end{align*}
      Let
\[
      V\indn =
      \begin{cases}
           V^1_{  R_1 \indn, R_2 \indn, R_3 \indn, R_4 \indn, \varepsilon \indn} & \text{ for $n$ even, and}
      \\
           V^2_{  R_1 \indn, R_2 \indn, R_3 \indn, R_4 \indn, \varepsilon \indn} & \text{ for $n$ odd.}
      \end{cases}
\]
      Accordingly, let
\[
      c\indn =
      \begin{cases}
           c^1_{  R_1 \indn, R_2 \indn, R_3 \indn, R_4 \indn, \varepsilon \indn} & \text{ for $n$ even, and}
      \\
           c^2_{  R_1 \indn, R_2 \indn, R_3 \indn, R_4 \indn, \varepsilon \indn} & \text{ for $n$ odd.}
      \end{cases}
\]
    Let $C \indX [0]=1$ and
\[
      C\indn =
      \begin{cases}
           \prod_{k=0}^{n-1} c \indX [k] & \text{ for $n>0$, and}
      \\
           \prod_{k=n}^{-1} \frac{1}{ c \indX [k] } & \text{ for $n<0$.}
      \end{cases}
\]
    Since $c \indX [k] \ge 1 - \varepsilon \indX [k] $
    and
    $\sum_{k\ge 0}  \varepsilon \indX [k] < \infty$, we have
\[
    C \indX [\infty] := \lim_{n\to\infty} C \indn \in ( 0, \infty)
    .
\]
    Define
\[
    V := \sum_{n\in \Z}  C \indn V \indn
    .
\]
    By~\eqref{eq:V-mass}, \eqref{eq:V-mass-B},
\begin{equation}\label{eq:Vn-mass}
         \frac \pi  2
                               ( (R_4 \indn )^2 -  (R_1 \indn)^2 )
      \le
         \mass ( V \indn )
      \le
          M \indn :=      2\pi ( (R_4 \indn )^2 -  (R_1 \indn)^2 )
     .
\end{equation}
    Since $C\indn$ is decreasing,
\begin{equation}\label{eq:conv}
      \sum_{n \ge -k }
         C \indn \mass ( V \indn )
     \le
      \sum_{n \ge -k }
         C \indX [-k] M \indn
     =
             C \indX [-k]
                2\pi  (R_4 \indX [-k] )^2
     < \infty
     .
\end{equation}
    $V$ is a Radon measure because, for every $k$,
\[
      V
            (  G_2(\{ x \setcolon \| x \| < 2^{k} \})   )
     \le
      \sum_{n \ge -k }
         C \indn \mass ( V \indn )
     < \infty
     .
\]
Obviously, the varifold $V$ is rectifiable.

    Using \eqref{eq:V-mass} and \eqref{eq:V-mass-B} more wisely than in \eqref{eq:Vn-mass}
    we get that
\begin{equation}\label{eq:unused--density-convergence}
         C \indX [\infty] (1 - \varepsilon \indn) \pi  R^2
      \le
           V ( G_2( \{ x \setcolon \| x \| \le R\} ) )
      \le
         C \indn (1 + \varepsilon \indn) \pi  R^2
\end{equation}
    whenever $R\in (0, R_4 \indn)$.
    Hence
    \[
        \theta^2(V, 0) = C \indX [\infty] \pi \in (0,\infty)
        .
    \]

\medbreak
{\bfseries 2. The varifold $V$ is stationary.}
        Let $X$ be a compactly supported smooth vector field on $\R^4$.
        Fix $k\in \N$ such that
        $
          \spt X \subset
              \{ x \setcolon \| x \| < 2^{k} \}
        $.
        We have
\[
         \left|
                \vari V \indn (X)
         \right|
        =
         \left|
              \int \diver_S X(x) \intd V \indn(x,S)
         \right|
        \le
           \| X \| _{C^1} \cdot \mass (V \indn)
        \le
           \| X \| _{C^1} \, C \indn M \indn
           .
\]
        Since $\sum_{n\ge -k} C \indn M \indn$ converges by \eqref{eq:conv}, we have
\begin{multline}\label{eq:V-vari}
        \vari V (X)
        =
              \int \diver_S X(x) \intd V(x,S)
\\
        =
           \sum_{n\ge -k}
              C \indn
              \int \diver_S X(x) \intd V \indn (x,S)
        =
           \sum_{n\ge -k}
              C \indn
                \vari V \indn (X)
        .
\end{multline}
       Next we use
       \eqref{eq:layer-V-vari} and \eqref{eq:layer-V-vari-B}
       to calculate
       $
           \sum_{n = -k}^{m}
                C \indn
                \vari V \indn (X)
       $.
       The first term is zero since
               \NewNewBsmall
               the support of $\vari V^{-k}$ is disjoint with
               \NewNewEsmall
       the support of $X$,
       next terms mutually cancel
       ($C \indn  c \indn = C \indn [+1]$, $R_1 \indn = R_4 \indn [+1]$)
       and
       what remains from
       the last one can be transformed so that we see
       that it
       converges to $0$. Formally,
\begin{align*}
           \sum_{n = -k}^{m}
                C \indn
                \vari V \indn (X)
           &=
           \sum_{n = -k}^{m}
              \left(
                     C \indn
                     \Mucircles_{R_4 \indn, \infty} (X)
                 -
                     C \indn
                     c \indn
                     \Mucircles_{R_1 \indn, \infty} (X)
              \right)
\\
           &=
                     C \indX [-k]
                     \Mucircles_{R_4 \indX [-k], \infty} (X)
                 -
                     C \indX [m]
                     c \indX [m]
                     \Mucircles_{R_1 \indX [m], \infty} (X)
\\
           &\overset{\hbox to 0pt{\hss$\scriptstyle \eqref{eq:Binfty}\hss$}}{
           =
           }
                 -
                     C \indX [m]
                     c \indX [m]
                \int\limits
                 _{F(( R_1 \indX [m]  \cdot S_1(\R^2))\times S_1(\R^2))}  X\cdot N
                             \ \frac{\intdx \, \Hausd^2}{ 2 \pi R_1 \indX [m]}
\\
           &\overset{\hbox to 0pt{\hss $\scriptstyle x=R_1 \indX [m] u$\hss}}{
           =
           }
           \ \ \   %
                 -
                     C \indX [m+1]
                \int\limits
                 _{F(S_1(\R^2)\times S_1(\R^2))}  X (R_1 \indX [m] u) \cdot N(u)
                             \ \frac{\intdx\, \Hausd^2 (u)}{ 2 \pi }
           \to 0
\end{align*}
      as $m\to \infty$
      since
      $ \lim R_1 \indX [m] = 0 $,
      $ \lim C \indX [m+1] = C \indX [\infty]$,
      and mainly
      $X(\rho u) \to X(0)$ uniformly as $\rho \to 0$
      and
\[
                \int\limits
                 _{F(S_1(\R^2)\times S_1(\R^2))}  N(u)
                             \ \frac{\intdx\, \Hausd^2 (u)}{ 2 \pi }
          =
                0
          .
\]
Therefore the sum in \eqref{eq:V-vari} is zero,
$\vari V(X)=0$ for arbitrary smooth compactly supported $X$,
and
$V$ is a stationary varifold.

\medbreak
{\bfseries 3. The tangents to $V$.}
        First we describe (without proof) the tangents to $V$:
        \[
            \VarTan_0 V
            =
              \{
                    \underbrace
                    {
                    (C \indX [\infty] /2 \pi)
                    \,
                    V_{\{\zeta\,R_1 \indX [-i] \}_{i\in \Z}}
                    }
                    _ { V _ \zeta }
                 \setcolon
                    \zeta > 0
              \}
        \]
        where $R_1 \indX [i]$ is as above
        and
        $V _ {\{r_i\}}$ as in \eqref{eq:Vseq}, \eqref{eq:V1def}, \eqref{eq:V2def}.
        Due to a ``periodicity'',
        $\zeta$ can be restricted to $ [R_1 \indX [0], R_1 \indX [-2] ) = [1,4)$.
        Then
        $V_\zeta$
        are mutually different
        and therefore not conical (cf.~Lemma~\ref{l:notcon,differ}).

        \NewNewBsmall
        (Recall that $V _ \zeta$ are $2$-varifolds supported by
        a $3$-dimensional cone in $\R^4$. In alternating
        layers, $V _ \zeta$ assume two different directions,
        namely those mentioned in \eqref{eq:differentSpan}.)\relax
        \NewNewEsmall

        To finish the formal proof of the theorem we do not need anything more than
        to pick out a single tangent varifold and show that it is not conical.
        Let $\lambda _ i = 4^{-i}$.
        Then $\lambda _ i R_1 \indn = R_1 \indn [+2i] $ and (see \eqref{eq:sptV-simple}, \eqref{eq:sptV-B-simple})
        \[
            \spt \left(
                      \eta_{0, \lambda_i\ \ImageVarifold } V \indn [+2i]
                 \right)
                 \subset
                     G_2( A_{R_1 \indn}^{R_4 \indn} )
                   \cap
                     G_{\RadAndX{D_{n+2i}}}^{\,\varepsilon \indn [+2i] }
        \]
        where $D_n$ is either symbol $\JOneThree$ ($n$ even) or $\JOneTwo$ ($n$ odd).
        Therefore $D_{n+2i} = D_{n}$ and
        \[
            \spt \left(
                      \eta_{0, \lambda_i\ \ImageVarifold } V
                 \right)
                 \subset
                  \bigcup_{n\in \Z}
                    \left(
                             G_2( A_{R_1 \indn}^{R_4 \indn} )
                           \cap
                             G_{\RadAndX{D_n}}^{\,\varepsilon \indn [+2i] }
                    \right)
            .
        \]
        From \eqref{eq:Vn-mass},
        \begin{multline}\label{eq:VR0-mass}
            \mass
              \left(
                 \left(
                      \eta_{0, \lambda_i\ \ImageVarifold } V
                 \right)
                 \mrest
                     G_2( A_{R_1 \indX [0] }^{R_4 \indX [0] } )
              \right)
           =
            (\lambda_i)^{-2}
            \mass
              \left(
                      V
                   \mrest
                     G_2( A_{R_1 \indX [2i] }^{R_4 \indX [2i] } )
              \right)
\\
           =
            (\lambda_i)^{-2}
            C  \indX [2i]
            \mass
                   (
                      V \indX [2i]
                   )
           \ge
            (\lambda_i)^{-2}
             C \indX [\infty]
             \frac \pi 2
             ( (R_4 \indX [2i] )^2
              -(R_1 \indX [2i] )^2)
           =
             \frac {3\pi} 2
             C \indX [\infty]
           .
        \end{multline}

        By the compactness theorem for Radon measures (\cite[p.~242, p.~22]{Simon}),
        there is a varifold $C$ and
        a subsequence of $\{\lambda_i\}$ (denoted by $\{\lambda_i\}$ again) such that
        $  \eta_{0, \lambda_i\ \ImageVarifold } V  \to  C  $.
        (We note without proof that
         in fact it is not necessary to pass to a subsequence
         since even the original sequence is convergent.)
        Hence $C \in \VarTan_0 V$.
        From the above,
        \begin{equation}\label{eq:C-mass}
            \mass
              \left(
                 C
                 \mrest
                     G_2( A_{R_1 \indX [0] }^{R_4 \indX [0] } )
              \right)
           \ge
             \frac {3\pi} 2
             C \indX [\infty]
           >
             0
        \end{equation}
        and
        \[
            \spt
                   C
                 \subset
                  \bigcup_{n\in \Z}
                    \left(
                             G_2( A_{R_1 \indn}^{R_4 \indn} )
                           \cap
                             G_{\RadAndX{D_n}}^{\, \varepsilon }
                    \right)
        \]
        for every $\varepsilon>0$ and thus also for $\varepsilon=0$.
        In particular
        \begin{align}\label{eq:dilatation1}
            \spt
                   C
                           \cap
                             G_2( \inter A_{R_1 \indX [0] }^{R_4 \indX [0]} )
                 &\subset
                             G_{\RadAndX{D_0}}^{\, 0 }
                  ,
        \\
        \label{eq:dilatation2}
            \spt
                   C
                           \cap
                             G_2( \inter A_{R_1 \indX [1] }^{R_4 \indX [1]} )
                 &\subset
                             G_{\RadAndX{D_{1}}}^{\, 0 }
        \end{align}
        where $\inter M$ denotes the interior of $M$.
        From \eqref{eq:dilatation2},
        \begin{equation}\label{eq:dilatation3}
            \spt
                   \left(
                      \eta_{0, 1/2\ \ImageVarifold }
                   C
                   \right)
                           \cap
                             G_2( \inter A_{R_1 \indX [0] }^{R_4 \indX [0]} )
                 \subset
                             G_{\RadAndX{D_{1}}}^{\, 0 }
            .
        \end{equation}
        Assume that $C$ is conical.
        Then
        $\spt C = \spt ( \eta_{0, 1/2\ \ImageVarifold } C )$.
        Since
        $ G_{\RadAndX{D_0}}^{\, 0 }$
        and
        $G_{\RadAndX{D_1}}^{\, 0 }$
        are disjoint
        (see \eqref{eq:J1324J1234}),
        we see that \eqref{eq:dilatation1} and  \eqref{eq:dilatation3} is possible only when
        \[
              \spt C
                           \cap
                             G_2( \inter A_{R_1 \indX [0] }^{R_4 \indX [0]} )
            =
              \emptyset
        \]
        which is a contradiction with \eqref{eq:C-mass} and \eqref{eq:nullsphere}.
        Hence $C$ is not conical.
\qed
\end{proof}
\begin{theorem}\label{thm:conical}
   There is a stationary rectifiable $2$-varifold
   $V$ in $\R^4$ that has
   at least two different
   conical tangents
   at\/ $0$
   and\/ $0<\theta^2(V, 0)<\infty$.
\end{theorem}
        \NewBrefsug
        The proof differs from the proof of the previous theorem mainly in a
        different definition of the sequences of radii $R_1, R_2, R_3, R_4$:
        taking $R_3/R_2$ large, the middle ``conical'' part becomes dominant.
        \NewErefsug
\begin{proof}
      For $n\in \Z$, let
\begin{align*}
      \varepsilon \indn &= 1/4(n^2+1)
      \\
      R_1 \indn &=
                         2^{-n^3}
      \\
      R_2 \indn &= (1+\varepsilon \indn ) R_1 \indn
      \\
      R_3 \indn &= (1-\varepsilon \indn ) R_4 \indn
      \\
      R_4 \indn &=
              R_1 \indn [-1]
      .
\end{align*}
 Note that $\{n^3\}$ is a strictly increasing sequence with
 increments at least one, hence
 $  R_1 \indn <  R_2 \indn <  R_3 \indn <  R_4 \indn $.
 Repeating the construction  of Theorem~\ref{thm:nonconical}
 we obtain
 a rectifiable stationary $2$-varifold $V$,
 but now the varifold's tangents at $0$ are different.

 Without proof we claim that,
 with $c = C \indX [\infty] / 2\pi$,
 $c V_{1,0,\infty}$ and $c V_{2,0,\infty}$
 (see Section~\ref{s:non-rect},
 \eqref{eq:V1def}, \eqref{eq:V2def})
 are two different (Lemma~\ref{l:notcon,differ}) conical tangent varifolds to $V$ at $0\in\R^4$.
 There are also tangent varifolds of the form
 $c ( V_{1,0,\rho} + V_{2,\rho,\infty} )$
 and
 $c ( V_{2,0,\rho} + V_{1,\rho,\infty} )$,
 $\rho > 0$; they are not conical, but they are ``conical near $0$''.\relax
 \footnote{We believe a slightly more complicated construction
 gives an example of a varifold whose all tangents are conical but
 the tangent at a point is non-unique.
 Basically, $\{
 \JOneThree, \JOneTwo
 \}$
 has to be replaced by a curve $\{ J(t) \setcolon t\in [0,1] \}$.
 A varifold would be used
 that takes directions in $G_{\RadAndX{J( j/2^k )}}^{\,1/(n^2+1)}$ on
 $A_{R_1 \indn}^{R_4 \indn} ( \R^4)$
 whenever $|n|=2^k + j > 2$, $k,j,\in \N$, $j \le 2^k$.}

 We will give the detailed proof for existence of two
 different conical tangent varifolds at $0$.
 Let
        $\lambda_i = i R_1 \indX [2i]$
 and
 $\tilde \lambda_i = i R_1 \indX [2i+1]$.

 Note that, for $i\to \infty$,
 $R_1  \indX [2i] / \lambda_i = 1/i \to 0$
 while
 $R_4  \indX [2i] / \lambda_i = 2^{-(2i-1)^3+(2i)^3} /i \to \infty$.
 We have
        \[
            \spt \left(
                      \eta_{0, \lambda_i\ \ImageVarifold } V \indX [2i]
                 \right)
                 \subset
                     G_2( A_{R_1 \indX [2i] / \lambda_i}^{R_4 \indX [2i] / \lambda_i} )
                   \cap
                     G_{\RadAndX{D_{2i}}}^{\,\varepsilon \indX [2i] }
        \]
        where $D_{2i} = D_0$ is the symbol  ``$\JOneThree$". Hence
        \begin{equation}\label{eq:spt-blowup-tg}
            \spt \left(
                      \eta_{0, \lambda_i\ \ImageVarifold } V
                 \right)
                 \subset
                             G_2( A_{0}^{R_1 \indX [2i] / \lambda_i} )
                  \cup
                    \left(
                             G_2( A_{R_1 \indX [2i] / \lambda_i}^{R_4 \indX [2i] / \lambda_i} )
                           \cap
                             G_{\RadAndX{D_0}}^{\,\varepsilon \indX [2i] }
                    \right)
                  \cup
                             G_2( A_{R_4 \indX [2i] / \lambda_i}^{\infty} )
            .
        \end{equation}

 As in the proof of the previous theorem,
 we pass to a subsequence (denoted by $\{\lambda_i\}$ again) if necessary,
 so that
 $\eta _{0,       \lambda_i, \ \ImageVarifold} V \to        C \in \VarTan_0 V$
 and
 $\eta _{0,\tilde \lambda_i, \ \ImageVarifold} V \to \tilde C \in \VarTan_0 V$.

    By \eqref{eq:spt-blowup-tg},
        \[
            \spt
                   C
                 \subset
                            G_2( \{0\} )
                         \cup
                            \bigcap _{ \varepsilon >0 }
                             G_{\RadAndX{D_0}}^{\,\varepsilon }
                 =
                            G_2( \{0\} )
                         \cup
                             G_{\RadAndX{D_0}}^{\, 0 }
            .
        \]
        By the same argument,
        \[
            \spt
                \tilde
                   C
                 \subset
                            G_2( \{0\} )
                         \cup
                             G_{\RadAndX{D_1}}^{\, 0 }
        \]
        where $D_1 = ``\JOneTwo\, "$.
        Hence $C=\tilde C$ is posssible (cf.\ again \eqref{eq:J1324J1234})
        only if $\spt C \cup \spt \tilde C \subset   G_2( \{0\} )$.
        However,
        for sufficiently large $i\in \N$
        we have
          $R_4 \indX [2i] / \lambda_i > 2$,
          $R_1 \indX [2i] / \lambda_i < 1$
        and,
        by \eqref{eq:V-mass} and \eqref{eq:V-mass-B},
        \begin{multline*}
                \mass ( ( \eta_{0, \lambda_i \ \ImageVarifold} V ) \mrest G_2 ( A _ {1} ^ {2} ) )
              =
                (\lambda_i)^{-2}
                \mass
                  \left
                  (
                        V \mrest G_2 ( A _ {\lambda_i} ^ {2\lambda_i} )
                  \right
                  )
 \\
              =
                (\lambda_i)^{-2}
                C \indX [2i]
                \mass
                  \left
                  (
                        V  \indX [2i] \mrest G_2 ( A _ {\lambda_i} ^ {2\lambda_i} )
                  \right
                  )
 \\
              \ge
                (\lambda_i)^{-2}
                C \indX [\infty]
                \frac \pi 2
                ( (2\lambda_i) ^ 2 - (\lambda_i) ^ 2)
                =
                \frac {3\pi} 2
                C \indX [\infty]
              >
                0
        \end{multline*}
        and therefore $C \neq \tilde C$ are two different conical tangents to $V$.
\qed
\end{proof}

\end{document}